\documentclass{amsart}
\usepackage{IMjournal,amsmath,amsthm,amssymb}
\usepackage[font=small]{subcaption}
\usepackage{url, enumitem, tikz, delimset}
\usepackage[hidelinks]{hyperref}
\usepackage{verbatim}

\newcommand{\interior}{\operatorname{int}}
\newcommand{\arcsinh}{\operatorname{arcsinh}}

\newcommand{\ow}{\overline{w}}
\newcommand{\uw}{\underline{w}}

\newcommand{\ov}{\overline{v}}

\newcommand{\oc}{\overline{c}}

\newcommand{\bU}{\mathbf{U}}
\newcommand{\bF}{\mathbf{F}}

\newcommand{\bW}{\mathbf{W}}

\newcommand{\bZ}{\mathbf{Z}}
\newcommand{\bz}{\mathbf{0}}

\newcommand{\calU}{\mathcal{U}}
\newcommand{\calK}{\mathcal{K}}
\newcommand{\calZ}{\mathcal{Z}}

\newcommand{\calW}{\mathcal{W}}

\newcommand{\R}{\mathbb R}
\newcommand{\N}{\mathbb N}
\newcommand{\Z}{\mathbb Z}

\renewcommand{\L}[1]{\mathbf{L^{#1}}}
\newcommand{\Lloc}[1]{\mathbf{L^{#1}_{loc}}}
\newcommand{\BV}{\mathbf{BV}}
\newcommand{\C}[1]{\mathbf{C^{#1}}}
\newcommand{\Cc}[1]{\mathbf{C_c^{#1}}}
\renewcommand{\d}{{\rm{d}}}
\DeclareMathOperator \tv {TV}
\newcommand{\TV}[1]{\tv[\,#1\,]}
\newcommand{\RS}{\mathsf{RS}}

\begin{document}


\title{On a $2\times2$ system of conservation laws\\with discontinuous flux}

\author{\|Felisia Angela |Chiarello|\ORC{0000-0002-5506-9086}, 
        \|Simone |Fagioli|\ORC{0000-0002-8581-1421}, 
        \|Bastien |Polizzi|\ORC{0000-0003-0535-1443}, 
        \|Massimiliano Daniele |Rosini|\ORC{0000-0003-2930-5740}}
\thanks
F.A.C., S.F.\ and M.D.R.\ are members of Gruppo Nazionale per l’Analisi Matematica, la Probabilit\`a e le loro Applicazioni (GNAMPA) of the Istituto Nazionale di Alta Matematica (INdAM). 
S.F.\ is partially supported by the Italian \lq\lq National Centre for HPC, Big Data and Quantum Computing\rq\rq\ - Spoke 5 \lq\lq Environment and Natural Disasters\rq\rq, by the InterMaths Network, \url{www.intermaths.eu}, by the INdAM-GNAMPA project 2026 code CUP E53C25002010001 \lq\lq Modelli di reazione-diffusione-trasporto: dall'analisi alle applicazioni\rq\rq. M.D.R.~acknowledges financial support from \lq\lq INdAM - GNAMPA Project\rq\rq, CUP~E5324001950001.
\endthanks

\rec{September 30, 2026}

\abstract 
We study a one-dimensional $2\times2$ system of conservation laws with discontinuous flux.
We prove the global existence of weak solutions for initial data with sufficiently small total variation. 
We construct explicitly the associated Riemann solver and analyze its main properties. 
Finally, we introduce a Glimm-type numerical scheme tailored to the exact Riemann solver and use it to reproduce the wave patterns obtained analytically. 
\endabstract

\keywords
   $2\times2$ systems of conservation laws; 
   non-symmetric Keyﬁtz-Kranzer systems; 
   discontinuous flux; 
   existence theory; 
   macroscopic vehicular traffic modelling
\endkeywords

\subjclass
35L65, 
35L45, 
35A01, 
90B20 
\endsubjclass

\section{Introduction}
We study a one-dimensional $2\times2$ non-symmetric Keyfitz-Kranzer system of conservation laws with discontinuous flux, see \cite{KeyfitzKranzer1980}. The model can be interpreted as a generalization of the Aw-Rascle-Zhang model \cite{AwRascle2000,Zhang2002} in the presence of spatially discontinuous road conditions.

We establish a global existence result for weak solutions for data with sufficiently small total variation. The result is obtained by embedding the problem into an augmented $3\times3$ strictly hyperbolic system and applying the general existence theory developed in \cite{AnconaMarson-2007}. This allows us to treat a rather general class of pressure laws $p$.

The case of a logarithmic pressure law has been studied in \cite{ChiarelloFagioliRosini1}. On the one hand, the wave-front tracking method exploited in \cite{ChiarelloFagioliRosini1} has the advantage of being constructive. On the other hand, it becomes technically challenging when the pressure law is changed. Here, instead, the existence is proved for more general pressure laws.

Our second aim is to construct the associated Riemann solver. In addition to the classical waves, an additional stationary discontinuity appears at the flux discontinuity points. We prove that the resulting solver is coherent and $\Lloc1$-continuous with respect to the initial data.

We complement the analytical study with a Glimm-type numerical approximation based on random choice methods, which are particularly well suited to problems in which contact discontinuities and interface waves must be captured sharply, see for instance \cite{chalonsgoatin, Collela_SIAJSSC_1982,Glimm1965,Toro1997}.

The paper is organized as follows. In Section~\ref{s:2} we introduce the system, its main properties, and prove the global existence result in Theorem~\ref{t:main}. Section~\ref{s:3} is devoted to the construction of the Riemann solver. In Section~\ref{s:4} we discuss representative Riemann problems and the corresponding wave patterns. Section~\ref{s:5} contains the numerical approximation and the comparison with the exact solutions. 

\section{The system and main properties}
\label{s:2}

The aim of this paper is the analysis of the Cauchy problem for a one-dimensional $2\times2$ non-symmetric Keyﬁtz-Kranzer type system of conservation laws with discontinuous flux
\begin{equation}
\left\{\begin{array}{@{}l@{}}
\partial_t \bU + \partial_x \left( c \, V\left(\frac{q}{\rho} - p(\rho)\right) \, \bU \right)=\bz,
\\[5pt]
\bU(0,x) = \bU_o(x),
\end{array}\right.
\label{e:2x2}
\end{equation}
for the conservative variable $\bU \doteq (\rho,q)$.
Here, the function $c \colon \R \to (0,\infty)$ is in $\BV$, $p \colon D_p \to \R$ and $V \colon D_V \to [0,\infty)$ are smooth functions on the intervals $D_p, D_V \subseteq [0,\infty)$.
We assume that $0\in D_V$ and that for any $\rho \in D_p\setminus\{0\}$ and $v \in D_V$
\begin{align}
\label{e:p}
&&
p'(\rho)&>0,&
\\
\label{e:V}
V(0)&=0,&
V'(v) &>0,&
\intertext{and the compatibility condition}
\label{e:pV1}
&&
\frac{V''(v)}{V'(v)} &< \frac{2p'(\rho)+\rho\,p''(\rho)}{\rho \, p'(\rho)^2}.
\end{align}
Take $I_p \doteq p(D_p)$, and observe that the inverse function $p^{-1} \colon I_p \to D_p$ is well defined and strictly increasing by \eqref{e:p}.

We let the conservative variables $\bU \doteq (\rho,q)$ vary in the domain
\begin{equation}
\label{e:U}
\calU \doteq \left\{ (\rho,q) \in \calU_\rho \times \calU_q : v(\rho,q) \in D_V , \ w(\rho,q) \in I_p \right\},
\end{equation}
where $\calU_\rho \subseteq D_p \setminus\{0\}$ and $\calU_q \subseteq \{ \rho\,w \in \R : \rho \in D_p\setminus\{0\}, \ w \in I_p\}$ are given intervals, and
\begin{align}
\label{e:vw}
v(\rho,q) &\doteq \frac{q}{\rho}-p(\rho),&
w(\rho,q) &\doteq \frac{q}{\rho}.
\end{align}
The eigenvalues and the corresponding eigenvectors of system \eqref{e:2x2} are
\begin{align}
\label{e:lam1}
\lambda_{1}(\bU,c) &= c \, \left( V\left(v(\bU)\right) - \rho \, p'(\rho) \, V'\left(v(\bU)\right) \right),&
r_{1}(\bU) &= \bU,
\\
\nonumber
\lambda_{2}(\bU,c) &= c \, V\left(v(\bU)\right),&
r_{2}(\bU) &= 
\left( 
\rho,
q + \rho^2 \, p'(\rho)
\right).
\end{align}
By \eqref{e:p} and \eqref{e:V}\textsubscript{$2$}, system \eqref{e:2x2} is strictly hyperbolic in $\calU$.
Moreover, the first and second characteristic fields are respectively genuinely nonlinear and linearly degenerate because 
\begin{align*}
D_{\bU}\lambda_1(\bU,c) \cdot r_1(\bU) &= \alpha(\bU,c),&
D_{\bU}\lambda_2(\bU,c) \cdot r_2(\bU) &= 0,
\end{align*}
where $\alpha(\bU,c) \doteq c \, \rho \, \left( \rho \, p'(\rho)^2 \, V''\left(v(\bU)\right) - \left( 2 p'(\rho) + \rho \, p''(\rho) \right) V'\left(v(\bU)\right) \right)$ does not vanish for any $\bU \in \calU$ by \eqref{e:pV1}.

The pair of functions $(v,w)$ given in \eqref{e:vw} forms a coordinate system of Riemann invariants for~\eqref{e:2x2} by \cite[Theorem~7.3.3]{Dafermos-book} because for any $\bU \in \calU$
\begin{align*}
D_{\bU}w (\bU) \cdot r_2(\bU) &= \rho \, p'(\rho) \neq 0,
&D_{\bU}w(\bU) \cdot r_1(\bU) &= 0,
\\
D_{\bU}v(\bU) \cdot r_1(\bU) &= -\rho \, p'(\rho) \neq 0,
&D_{\bU}v(\bU) \cdot r_2(\bU) &= 0,
\end{align*}
by \eqref{e:p}.
In particular $w$ and $v$ are respectively the first and second \emph{Riemann invariants}, see \cite[Definition~7.3.1]{Dafermos-book}.
Notice that the eigenvectors and the Riemann invariants do not depend on $c$.

For completeness, we recall the following definition.
\begin{definition}
\label{d:weak_sol}
Fix $\bU_o = (\rho_o,q_o) \in \L{\infty}(\R; \calU)$.
A function $\bU$ in $\L\infty(\R_+ \times \R; \calU)\cap \C0(\R_+ ;\Lloc{1}(\R; \calU))$ is a weak solution of the Cauchy problem \eqref{e:2x2} if for any test function $\varphi \in \Cc{\infty} \left( [0,\infty)\times \R; \R\right)$ we have
\begin{equation*}
\int_{\R_+}\int_\R \left( \partial_t\varphi + c \, V(v(\bU)) \, \partial_x\varphi\right) \, \bU \, \d x \, \d t + \int_\R \bU_o(x) \, \varphi(0,x) \, \d x = \bz.
\end{equation*}
\end{definition}
We are now in a position to state our existence result, which holds in the interior of the domain $\calU$ given in \eqref{e:U}, namely in
\begin{equation}
\label{e:Ue}
\calU_e \doteq \interior(\calU),
\end{equation}
and under the further compatibility condition: for any $\rho \in D_p\setminus\{0\}$ and $v\in D_V$ we have
\begin{equation}
\label{e:pV2}
\frac{V(v)}{V'(v)} < \rho \, p'(\rho).
\end{equation}

\begin{theorem}
\label{t:main}
Assume that $p \colon D_p \to \R$ and $V \colon D_V \to [0,\infty)$ are $\C2$ and satisfy the conditions listed in \eqref{e:p}, \eqref{e:V}, \eqref{e:pV1} and \eqref{e:pV2}.
Then, for any given compact subsets $\calK_c$ of $(0,\infty)$ and $\calK_{\bU}$ of $\calU_e$ given in \eqref{e:Ue}, there exists a constant $\delta>0$ such that, for every coefficient function $c \in \BV\cap\Lloc1(\R;(0,\infty))$ and initial datum $\bU_o \in \Lloc1(\R;\calU_e)$ satisfying the following conditions
\begin{align}
\label{e:Uc_cond}
&\TV{\bU_o} + \TV{c} \leq \delta,&
&\lim_{x\to-\infty} c(x) \in \calK_c,&
&\lim_{x\to-\infty} \bU_o(x) \in \calK_{\bU},
\end{align}
the corresponding Cauchy problem \eqref{e:2x2} admits a globally defined weak solution taking values in $\calU_e$.
\end{theorem}

\begin{proof}
Consider the Cauchy problem for a $3\times3$ system of conservation laws
\begin{equation}
\label{e:3x3}
\begin{cases}
\partial_t\bZ+\partial_xG(\bZ)=\bz,
\\
\bZ(0,x)=\bZ_o(x),
\end{cases}
\end{equation}
where both the conserved quantity $\bZ \doteq (\bU,\zeta) = (\rho,q,\zeta)$ and the initial datum $\bZ_o$ take values in the \emph{open} domain
\begin{equation}
\label{e:Ze}
\calZ_e \doteq \{ (\rho,q,\zeta) \in \R^3 : \rho \in \interior(D_p),\ v(\rho,q) \in \interior(D_V),\ \zeta>0 \},
\end{equation}
and the flux $G \colon \calZ_e \to \R^3$ is defined by $G(\bU,\zeta) \doteq 
\zeta \, V\left( v(\bU) \right) 
(\bU,0)^T$, where $v$ is defined in \eqref{e:vw}\textsubscript{$1$}.
Observe that $\bz \notin \calZ_e$.
Moreover, $G$ is $\C2$ in $\calZ_e$.
The eigenvalues and the corresponding eigenvectors of the system of conservation laws in \eqref{e:3x3} are
\begin{gather*}
\lambda_1(\bZ)=\zeta \, \left(V\left(v(\bU)\right) - \rho \, p'(\rho) \, V'\left(v(\bU)\right) \right),
\quad
R_1(\bZ)=\left(\begin{array}{@{}c@{}}
\bU \\
0 \\
\end{array}\right),
\\
\lambda_c(\bZ)=0,
\quad
R_c(\bZ)=\left(\begin{array}{@{}c@{}}
\bU \\
\zeta \, \left( \frac{\rho \, p'(\rho) \, V'\left(v(\bU)\right)}{V\left(v(\bU)\right)} - 1 \right) \\
\end{array}\right),
\\
\lambda_2(\bZ)=\zeta \, V\left(v(\bU)\right),
\qquad
R_2(\bZ)=\left(\begin{array}{@{}c@{}}
\rho \\
q + \rho^2 \, p'(\rho) \\
0 \\
\end{array}\right).
\end{gather*}
System \eqref{e:3x3} is strictly hyperbolic in $\calZ_e$ because $\lambda_1(\bZ) < 0 = \lambda_c(\bZ) <  \lambda_2(\bZ)$ by \eqref{e:V} and \eqref{e:pV2}. 
The characteristic fields related to $\lambda_c$ and $\lambda_2$ are linearly degenerate because $\nabla\lambda_c\cdot r_c \equiv 0$ and $\nabla\lambda_2\cdot r_2 \equiv 0$.
The characteristic field corresponding to $\lambda_1$ is genuinely non-linear because by \eqref{e:pV1} we have
\[
\nabla\lambda_1(\bZ) \cdot r_1(\bZ) = 
\zeta \, \rho \, \left(
\rho \, p'(\rho)^2 \, V''\left(v(\bU)\right) 
- \left( 2 p'(\rho) + \rho \, p''(\rho) \right) V'\left(v(\bU)\right) 
\right) \neq 0.
\]

System \eqref{e:3x3} satisfies the hypotheses of \cite[Theorem~2]{AnconaMarson-2007}, hence for any given compact subset $\calK$ of $\calZ$, there exists a constant $\delta>0$ such that, for every initial datum $\bZ_o \in \Lloc1(\R)$ satisfying the following conditions
\begin{align}
\label{e:Z_cond}
&\TV{\bZ_o} \leq \delta,&
&\lim_{x\to-\infty} \bZ_o(x) \in \calK,
\end{align}
the Cauchy problem~\eqref{e:3x3} admits a weak solution defined globally in time.

We now take $\bZ_o = (\bU_o,c) = (\rho_o, q_o, c)$.
By \eqref{e:Uc_cond} we have that $\bZ_o \in \Lloc1(\R)$ satisfies \eqref{e:Z_cond} with $\calK = \calK_{\bU} \times \calK_c$.
Furthermore, $\bZ_o$ takes values in $\calZ_e$ defined in \eqref{e:Ze} because $\bU_o$ takes values in $\calU_e$ given in \eqref{e:Ue} and $c$ in $(0,\infty)$.
The Cauchy problem \eqref{e:3x3} corresponding to such initial datum then admits a weak solution $\bZ = (\bU,\zeta) = (\rho,q,\zeta)$ defined for all $t\geq0$, which then satisfies the following condition: for any test function $\varphi \in \Cc\infty([0,\infty)\times\R;\R)$
\begin{align*}
&\int_{\R_+} \int_{\R} \left( \bU \, \partial_t\varphi + \zeta \, V\left(\frac{q}{\rho}-p(\rho)\right) \, \bU \, \partial_x\varphi \right) \d x \, \d t + \int_{\R} \bU_o(x) \, \varphi(0,x) \d x = \bz,\\
&\int_{\R_+} \int_{\R} \zeta \, \partial_t\varphi \, \d x \, \d t + \int_{\R} c(x) \, \varphi(0,x) \d x = 0.
\end{align*}
The last condition implies that $\zeta \equiv c$ a.e.~in $\R_+ \times \R$.
Therefore, by substituting $c$ for $\zeta$ in the first two equations, we deduce that $\bU = (\rho,q)$ is a weak solution of \eqref{e:2x2} in the sense of Definition~\ref{d:weak_sol} and taking values in $\calU_e$.
\end{proof}

\section{The Riemann solver}
\label{s:3}

In this section we assume that $p$ and $V$ satisfy \eqref{e:p}, \eqref{e:V} and \eqref{e:pV1}.
For simplicity in the exposition, we consider the coordinate system $\bW \doteq (\rho,w)$.
With a slight abuse of notation, we define
\begin{equation}
\label{e:v}
v(\rho,w) \doteq w-p(\rho).
\end{equation}
Away from the vacuum, the system of conservation laws in \eqref{e:2x2} is equivalent to
\begin{subequations}
\label{e:Riemann}
\begin{equation}
\left\{\begin{array}{@{}l@{}}
\partial_t \rho + \partial_x \left( c \, V\left(v(\rho,w)\right) \, \rho \right)=0,
\\[5pt]
\partial_t (\rho \, w) + \partial_x \left(c \, V\left(v(\rho,w)\right) \rho \, w\right)=0.
\end{array}\right.
\label{e:system}
\end{equation}
Unlike \eqref{e:2x2}, the above system of conservation laws allows for solutions involving a vacuum state whenever $0 \in D_p$.
The aim of this section is to investigate the Riemann problem associated with \eqref{e:system}, namely the Cauchy problem for \eqref{e:system} with initial condition
\begin{align}
&\bW(0,x) = \bW_o(x),&
&\bW_o(x) \doteq \begin{cases}
\bW_L&\hbox{if } x<0,\\
\bW_R&\hbox{if } x\geq0,
\end{cases}
\label{e:iniRiemann}
\end{align}
and with
\begin{equation}
\label{e:cRiemann}
c(x) = \begin{cases}
c^-&\hbox{if } x<0,\\
c^+&\hbox{if } x\geq0,
\end{cases}
\end{equation}
\end{subequations}
where $c^\pm > 0$, $\bW_L=(\rho_L,w_L) \in \calW(\ov^-)$ and $\bW_R = (\rho_R,w_R) \in \calW(\ov^+)$ are given constants.
Here, the domains $\calW(\ov^\pm)$ are defined as follows.
Fix $\uw, \ow \in \R$ and $\ov^\pm \in (0,\infty)$ so that
\begin{align}
\label{e:DevinTownsend}
&\uw < \ow,&
&\uw - \ov^\pm \in I_p,&
&\ow \in I_p,&
&\ov^\pm \in D_V.
\end{align}
We then take
\[\calW(\ov^\pm) \doteq \left\{ (\rho,w) \in \left[ p^{-1}(\uw-\ov^\pm) , p^{-1}(\ow) \right] \times \left[ \uw , \ow \right] : v(\rho,w) \in \left[ 0 , \ov^\pm \right]\right\}.\]
For later use, take
\begin{equation}
\label{e:ov}
\ov \doteq \max\{\ov^-,\ov^+\}.
\end{equation}

\begin{remark}
If $(\rho,w) \in \calW(\ov^\pm)$ then $\rho \in D_p$, $w-v(\rho,w) \in I_p$ and $\rho \in [0,p^{-1}(w)]$.
Indeed, the first condition follows by observing that $\left[ p^{-1}(\uw-\ov^\pm) , p^{-1}(\ow) \right] \subseteq D_p$ by \eqref{e:DevinTownsend}\textsubscript{$2,3$}.
The second condition follows by observing that $w-v(\rho,w) = p(\rho) \in I_p$ by \eqref{e:v} and the fact that $\rho \in D_p$.
The last condition follows from the following argument.
The hypothesis $v(\rho,w) \in [0,\ov^\pm]$ implies that $v(\rho,w) \geq 0$, then by the monotonicity of $p$, see \eqref{e:p}, we have that $\rho \leq p^{-1}(w)$.
We conclude then by observing that $\rho \in D_p \subseteq [0,\infty)$.
\end{remark}

\begin{remark}
If a vehicle is characterized by the Lagrangian marker $w$, then it has length $1/p^{-1}(w)$.
Indeed, if vehicles with the same Lagrangian marker $w$ are bumper-to-bumper, then their velocity is zero, $V(v)=0$.
By \eqref{e:V} this corresponds to $v=0$ and then by \eqref{e:v} to the density $\rho = p^{-1}(w)$. Hence an interval of length $\ell = 1/p^{-1}(w)$ contains $\rho \, \ell = 1$ vehicle.
\end{remark}

\begin{remark}
We have that a vacuum state belongs to $\calW(\ov^\pm)$ if and only if $0\in I_p$ and $p^{-1}(\uw-\ov^\pm) = 0$.
In this case, the unique vacuum state in $\calW(\ov^\pm)$ is $\bW =  (0,\uw)$.
\end{remark}

For completeness, we recall the following definition.
\begin{definition}
\label{d:weak_solW}
Fix $(\bW_L,\bW_R) \in \calW(\ov^-) \times \calW(\ov^+)$.
A function $\bW=(\rho,w)$ in $\L\infty((0,\infty);\BV(\R))\cap \C0(\R_+ ;\Lloc{1}(\R))$ is weak solution of the Riemann problem \eqref{e:Riemann} if for any test function $\varphi \in \Cc{\infty} \left( [0,\infty)\times \R; \R\right)$ we have
\begin{align*}
&\int_{\R_+}\int_\R \left( \partial_t\varphi + c \, V(w-p(\rho)) \, \partial_x\varphi\right) \, \rho \, \d x \, \d t + \int_\R \rho_o(x) \, \varphi(0,x) \, \d x = 0,
\\
&\int_{\R_+}\int_\R \left( \partial_t\varphi + c \, V(w-p(\rho)) \, \partial_x\varphi\right) \, \rho \, w \, \d x \, \d t + \int_\R \rho_o(x) \, w_o(x) \, \varphi(0,x) \, \d x = 0.
\end{align*}
\end{definition}

We follow the argument in \cite{RosiniAnnales} to give the definition of our Riemann solver.
We first need to introduce some notation.
Let $D_*^\pm \doteq \calW(\ov^\pm) \times [\uw,\ow]$ and define $\bW_* \colon D_*^\pm \to \calW(\ov^\pm)$ as follows: for any $(\bW_R,w_L) \in D_*^\pm$ we let
\begin{equation}
\label{e:w*}
\begin{aligned}
\rho_*(\bW_R,w_L) &\doteq \inf\left\{ \rho \in [p^{-1}(w_L-\ov^+),p^{-1}(w_L)] : v(\rho,w_L) \leq v(\bW_R) \right\},
\\
w_*(\bW_R,w_L) &\doteq w_L.
\end{aligned}
\end{equation}
Observe that $\bW_*$ is well defined.
Indeed, by definition we readily have $w_*(\bW_R,w_L) = w_L \in [\uw,\ow]$ and $v\left(\bW_*(\bW_R,w_L)\right) = v(\bW_R) \in [0,\ov^\pm]$.
Then observe that $p\left(\rho_*(\bW_R,w_L)\right) = w_L - v(\bW_R) \in [\uw-\ov^\pm,\ow]$ and therefore $\rho_*(\bW_R,w_L) \in [p^{-1}(\uw-\ov^\pm),p^{-1}(\ow)]$.

For any $w \in [\uw,\ow]$, we introduce the function
\begin{equation}
\label{e:F}
\begin{array}{@{}r@{}r@{\,}c@{\,}l@{}}
F_{w} \colon & D_{F_{w}} \doteq \left[ p^{-1}(w-\ov) , p^{-1}(w) \right]& \to  &\left[ 0 , \infty \right),
\\
&\rho&\mapsto& V\left( v(\rho,w) \right) \, \rho,
\end{array}
\end{equation}
with $\ov$ given in \eqref{e:ov}.
Observe that $v(\rho,w) \in [0,\ov]$ for every $w \in [\uw,\ow]$ and $\rho \in D_{F_{w}}$.
Notice that the flux of $\rho$ in \eqref{e:system} is $c\,F_{w}(\rho)$.
We collect the main properties of $F_{w}$ in the following lemma.

\begin{lemma}
For any $w \in [\uw,\ow]$, the function $F_{w}(\rho) \doteq \rho \, V\left( v(\rho,w) \right)$ given in \eqref{e:F} is strictly concave down and $F_{w}\left(p^{-1}( w )\right) = 0$.
\end{lemma}
\begin{proof}
By \eqref{e:pV1} we have
\begin{align}
\label{e:Fwprimo}
F_{w}'(\rho) &= V\left( w - p(\rho) \right) - \rho \, p'(\rho) \, V'\left( w - p(\rho) \right),
\\
\label{e:Fwsecondo}
F_{w}''(\rho) &= -V'\left( w - p(\rho) \right) \, \left( 2p'(\rho)+\rho\,p''(\rho) \right) + \rho \, p'(\rho)^2 \, V''\left( w - p(\rho) \right) < 0,
\end{align}
hence $F_{w}$ is strictly concave down.
The latter claim follows from \eqref{e:v} and \eqref{e:V}\textsubscript{$1$}.
\end{proof}

For any $w \in [\uw,\ow]$, we also introduce the function 
\begin{equation}
\label{e:lambdaw}
\begin{array}{@{}r@{}r@{\,}c@{\,}l@{}}
\lambda_w \colon & D_{F_{w}}& \to  &\R,
\\
&\rho&\mapsto&V\left( w - p(\rho) \right) - \rho \, p'(\rho) \, V'\left( w - p(\rho) \right).
\end{array}
\end{equation}
By definition we have $c \, \lambda_w(\rho) = \lambda_1(\rho,w\,\rho,c)$, where $\lambda_1$ is given in \eqref{e:lam1}\textsubscript{$1$}.
Observe that $\lambda_w(\rho) = {F_{w}}'(\rho)$ by \eqref{e:Fwprimo}, hence it is strictly decreasing by \eqref{e:Fwsecondo}, and therefore it is invertible.
Consider its inverse function $\mathcal{H}_w \colon \lambda_w\left( D_{F_{w}} \right) \to D_{F_{w}}$.

Finally, introduce $s \colon  \{(\bW_-,\bW_+) \in \calW(\ov)^2 : \rho_- \neq \rho_+\} \to \R$ defined by
\begin{equation}
\label{e:s}
s(\bW_-,\bW_+) \doteq \frac{F_{w_+}(\rho_+)-F_{w_-}(\rho_-)}{\rho_+-\rho_-}.
\end{equation}
Observe that if a weak solution of \eqref{e:Riemann} has a discontinuity $(\bW_-,\bW_+) \in \calW(\ov^\pm)^2$ at $\pm x>0$, then it has speed of propagation $c^\pm \, s(\bW_-,\bW_+)$.

Consider the one-parameter family of \emph{fundamental diagrams} associated to $c^\pm > 0$
\[\{(\rho,f) \in D_p \times [0,\infty) : (\rho,w) \in \calW(\ov^\pm), \ f = c^\pm \, F_{w}(\rho) \} \qquad\hbox{for } w\in [\uw,\ow].\]
Observe that the fundamental diagrams associated with $c^- = c^+$ do not intersect. On the contrary, when $c^- \neq c^+$ this property is lost if the fundamental diagrams corresponding to $c^-$ and $c^+$ are represented in the same plot.
Therefore, for simplicity of exposition, we construct our Riemann solver first for the case $c^- = c^+$, and then for the case $c^-\neq c^+$.

\subsubsection*{Basic case}

Assume that $c^\pm = \bar{c}$.
Take $\ov^\pm = \ov$ and $\calW \doteq \calW(\ov)$.
The Riemann solver associated to the Riemann problem \eqref{e:Riemann} is then given as follows.

\begin{definition}\label{d:Peryphery}
Fix $\bar{c}>0$. 
The Riemann solver $\RS_{\bar{c}} \colon \calW^2 \to \BV(\R;\calW)$ for the Riemann problem \eqref{e:Riemann} with $c^{\pm}=\bar{c}$ is defined as follows:
\begin{enumerate}[label={{\rm\bf(L.\arabic*)}}, leftmargin=*]

\item\label{L-C}
If $\bW_L,\bW_R \in \calW$ with $w_L \neq w_R$ and $v(\bW_L) = v(\bW_R) = v$, then
\[
\RS_{\bar{c}}[\bW_L,\bW_R](\nu) \doteq
\begin{cases}
\bW_L&\hbox{if }\nu/\bar{c} < V(v),
\\
\bW_R&\hbox{if }\nu/\bar{c} \geq V(v)).
\end{cases}
\]

\item\label{L-S}
If $\bW_L,\bW_R \in \calW$ with $w_L = w_R$ and $v(\bW_L) > v(\bW_R)$, then
\[
\RS_{\bar{c}}[\bW_L,\bW_R](\nu) \doteq
\begin{cases}
\bW_L&\hbox{if } \nu/\bar{c} < s(\bW_L,\bW_R),
\\
\bW_R&\hbox{if } \nu/\bar{c} \geq s(\bW_L,\bW_R),
\end{cases}
\]
where $s$ is given in \eqref{e:s}.

\item\label{L-R}
If $\bW_L,\bW_R \in \calW$ with $w_L = w_R = w$ and $v(\bW_L) < v(\bW_R)$, then
\[
\RS_{\bar{c}}[\bW_L,\bW_R](\nu) \doteq
\begin{cases}
\bW_L&\hbox{if } \nu/\bar{c} < \lambda_{w}(\rho_L),
\\
\left(\mathcal{H}_w\left(\nu/\bar{c}\right) , w\right)&\hbox{if } \lambda_{w}(\rho_L) \leq \nu/\bar{c} < \lambda_{w}(\rho_R),
\\
\bW_R&\hbox{if }\nu/\bar{c} \geq \lambda_{w}(\rho_R),
\end{cases}
\]
where $\lambda_w$ is given in \eqref{e:lambdaw}.

\item\label{L-SC}
If $\bW_L,\bW_R \in \calW$ with $w_L\neq w_R$ and $v(\bW_L) > v(\bW_R)$, then
\[
\RS_{\bar{c}}[\bW_L,\bW_R](\nu) \doteq
\begin{cases}
\bW_L&\hbox{if } \nu/\bar{c} < s(\bW_L,\bW_M),
\\
\bW_M&\hbox{if } s(\bW_L,\bW_M)\leq \nu/\bar{c} < V(v(\bW_R)),
\\
\bW_R&\hbox{if } \nu/\bar{c} \geq V(v(\bW_R)),
\end{cases}
\]
where $\bW_M  \doteq \bW_*(\bW_R,w_L)$.
\item\label{L-RC}
If $\bW_L,\bW_R \in \calW$ with $w_L\neq w_R$ and $v(\bW_L) < v(\bW_R)$, then
\[
\RS_{\bar{c}}[\bW_L,\bW_R](\nu) \doteq
\begin{cases}
\bW_L&\hbox{if }\nu/\bar{c} < \lambda_{w_L}(\rho_L),
\\
\left(\mathcal{H}_{w_L}\left(\nu/\bar{c}\right),w_L\right)&\hbox{if } \lambda_{w_L}(\rho_L) \leq \nu/\bar{c} < \lambda_{w_L}(\rho_M),
\\
\bW_M&\hbox{if } \lambda_{w_L}(\rho_M) \leq \nu/\bar{c} < V(v(\bW_R)),
\\
\bW_R&\hbox{if } \nu/\bar{c} \geq V(v(\bW_R)),
\end{cases}
\]
where $\lambda_w$ is given in \eqref{e:lambdaw} and $\bW_M  \doteq \bW_*(\bW_R,w_L)$.
\item
If $\bW_L=\bW_R$, then $\RS_{\bar{c}}[\bW_L,\bW_R] \,\dot{\equiv}\, \bW_{L,R}$.
\end{enumerate}
\end{definition}

\begin{proposition}
The Riemann solver given in Definition~\ref{d:Peryphery} is well defined.
\end{proposition}
\begin{proof}
By the monotonicity properties of $\mathcal{H}_w$, it is sufficient to observe that the intermediate state $\bW_M$ belongs to $\calW$, both in the case \ref{L-SC} and \ref{L-RC}.
In both cases, we let $\bW_M = \bW_*(\bW_R,w_L)$ and we already proved that $\bW_*$ takes values in $\calW$.
This concludes the proof.
\end{proof}

\noindent
Some comments on the above definition are in order.
In case \ref{L-C} we have that $\RS_{\bar{c}}[\bW_L,\bW_R]$ is the $2$-contact discontinuity $\mathcal{C}_2(\bW_L,\bW_R)$.
In case \ref{L-S} we have $0 \leq \rho_L < \rho_R \leq p^{-1}(w_{L,R})$ and $\RS_{\bar{c}}[\bW_L,\bW_R]$ is the $1$-shock $\mathcal{S}_1(\bW_L,\bW_R)$.
In case \ref{L-R} we have $0 \leq \rho_R < \rho_L \leq p^{-1}(w_L)$ and $\RS_{\bar{c}}[\bW_L,\bW_R]$ is the $1$-rarefaction $\mathcal{R}_1(\bW_L,\bW_M)$.
In case \ref{L-SC} we have $0 \leq \rho_L < \rho_M \leq p^{-1}(w_L)$ and $\RS_{\bar{c}}[\bW_L,\bW_R]$ is the juxtaposition of $\mathcal{S}_1(\bW_L,\bW_M)$ and $\mathcal{C}_2(\bW_M,\bW_R)$.
In case \ref{L-RC} we have $0 \leq \rho_M < \rho_L \leq p^{-1}(w_L)$, with $\rho_M=0$ if and only if $v(\bW_R) \geq w_L$, and $\RS_{\bar{c}}[\bW_L,\bW_R]$ is the juxtaposition of $\mathcal{R}_1(\bW_L,\bW_M)$ and $\mathcal{C}_2(\bW_M,\bW_R)$.

\begin{proposition}
\label{p:SnarkyPuppy}
For any $\bW_L,\bW_R\in\calW$ the map $(t,x) \mapsto \RS_{\bar{c}}[\bW_L,\bW_R](x/t)$ is a weak solution of the Riemann problem \eqref{e:Riemann} in the sense of Definition~\ref{d:weak_solW}.
\end{proposition}
\begin{proof}
It is sufficient to observe that any discontinuity satisfies the Rankine-Hugoniot condition.
This readily follows from the expression of the propagation speed \eqref{e:s}, and by observing that $V\left( v(\rho,w) \right) \, \rho$ is continuous at $\nu=0$.
\end{proof}

\subsubsection*{General case}

Assume that $c^- \neq c^+$.
In this case, the choice of $\ov^\pm$ to get a well defined Riemann solver is a delicate issue and will be discussed later, see Proposition~\ref{p:OrbitCulture}.
The Riemann solver associated to the Riemann problem \eqref{e:Riemann} is given as follows.
We first need to introduce some notation.
Denote by $\RS_\pm$ and $\RS$ the Riemann solvers given in Definition~\ref{d:Peryphery}, corresponding to $\oc=c^{\pm}$ and $\oc=1$, respectively.
Observe that for a fixed $w \in [\uw,\ow]$, the quantities
\[\mathsf{F}_\pm(w) \doteq \max\{F_{w}(\rho) : \rho \in D_{F_{w}}, v(\rho,w) \in [0,\ov^\pm] \}\]
are nonnegative and bounded.
Moreover $\mathsf{r}_\pm(w) \in D_{F_{w}}$ such that $F_{w}\left(\mathsf{r}_\pm(w)\right) = \mathsf{F}_\pm(w)$ are uniquely determined.
Define $\mathsf{Q}_- \colon \calW(\ov^-) \to [0,\infty)$ and $\mathsf{Q}_+ \colon \calW(\ov^+) \times [\uw,\ow] \to [0,\infty)$ as follows:
\begin{align}
\label{e:Qm}
\mathsf{Q}_-(\bW_L) &\doteq 
\begin{cases}
\mathsf{F}_-(w_L)&\hbox{if } \rho_L \geq \mathsf{r}_-(w_L),
\\
F_{w_L}(\rho_L)&\hbox{if } \rho_L < \mathsf{r}_-(w_L),
\end{cases}
\\
\label{e:Qp}
\mathsf{Q}_+(\bW_R,w_L) &\doteq 
\begin{cases}
F_{w_L}\left(\rho_*(\bW_R,w_L)\right)&\hbox{if } \rho_*(\bW_R,w_L) > \mathsf{r}_+(w_L),
\\
\mathsf{F}_+(w_L)&\hbox{if } \rho_*(\bW_R,w_L) \leq \mathsf{r}_+(w_L).
\end{cases}
\end{align}
Then define $\mathsf{Q} \colon \calW(\ov^-)\times\calW(\ov^+) \to [0,\infty)$ given by
\begin{equation}
\label{e:Q}
\mathsf{Q}(\bW_L,\bW_R) \doteq \min\left\{
c^- \, \mathsf{Q}_-(\bW_L),
c^+ \, \mathsf{Q}_+(\bW_R,w_L)
\right\}.
\end{equation}
\begin{definition}
\label{d:hat}
Fix $(\bW_L,\bW_R) \in \calW(\ov^-)\times\calW(\ov^+)$ and define $\hat{\bW}(\bW_L,\bW_R) \in \calW(\ov^-)$ as follows:
\begin{enumerate}[leftmargin=*, label={{\bf(H.\roman*)}}]
\item 
\label{i:trivhat}
If $\mathsf{Q}(\bW_L,\bW_R) = c^- \, F_{w_L}(\rho_L)$, then we let $\hat{\bW}(\bW_L,\bW_R) = \bW_L$.
\item 
\label{i:hat}
If $\mathsf{Q}(\bW_L,\bW_R) \neq c^- \, F_{w_L}(\rho_L)$ and $\bW \in \calW(\ov^-)$ such that
\begin{align}
\label{e:hat}
&\rho \in [\mathsf{r}_-(w_L),p^{-1}(w_L)],&
&w=w_L,&
&c^- \, F_{w_L}(\rho) = \mathsf{Q}(\bW_L,\bW_R),
\end{align}
exists, then we denote it by $\hat{\bW}(\bW_L,\bW_R)$.
\end{enumerate}
\end{definition}

\begin{definition}
\label{d:check}
Fix $(\bW_L,\bW_R) \in \calW(\ov^-)\times\calW(\ov^+)$ and define $\check{\bW}(\bW_L,\bW_R) \in \calW(\ov^+)$ as follows:
\begin{enumerate}[leftmargin=*, label={{\bf(C.\roman*)}}]
\item 
\label{i:trivcheck}
If $\mathsf{Q}(\bW_L,\bW_R) = c^+ \, F_{w_L}\left(\rho_*(\bW_R,w_L)\right)$, then we let $\check{\bW}(\bW_L,\bW_R) = \bW_*(\bW_R,w_L)$.
\item 
\label{i:check}
If $\mathsf{Q}(\bW_L,\bW_R) \neq c^+ \, F_{w_L}\left(\rho_*(\bW_R,w_L)\right)$ and $\bW \in \calW(\ov^+)$ such that
\begin{align}
\label{e:check}
&\rho \in [0,\mathsf{r}_+(w_L)],&
&w=w_L,&
&c^+ \, F_{w_L}(\rho) = \mathsf{Q}(\bW_L,\bW_R),
\end{align}
exists, then we denote it by $\check{\bW}(\bW_L,\bW_R)$.
\end{enumerate}
\end{definition}

Introduce the domain
\[D_{\RS} \doteq \left\{ (\bW_L,\bW_R) \in \calW(\ov^-)\times\calW(\ov^+) : \hat{\bW}, \check{\bW} \hbox{ are well defined at } (\bW_L,\bW_R) \right\},\]
and observe that for any $(\bW_L,\bW_R) \in D_{\RS}$ we have
\begin{equation*}
c^- \, F_{\hat{w}(\bW_L,\bW_R)}\left(\hat{\rho}(\bW_L,\bW_R)\right) = \mathsf{Q}(\bW_L,\bW_R)  = c^+ \, F_{\check{w}(\bW_L,\bW_R)}\left(\check{\rho}(\bW_L,\bW_R)\right).
\end{equation*}

We are now in the position to give the following definition.

\begin{definition}\label{d:Haken}
For any $(\bW_L,\bW_R) \in D_{\RS}$ we let
\begin{equation*}
\RS_{c^-,c^+}[\bW_L,\bW_R](\nu) \doteq
\begin{cases}
\RS_-[\bW_L,\hat{\bW}(\bW_L,\bW_R)](\nu)&\hbox{if }\nu < 0,
\\
\RS_+[\check{\bW}(\bW_L,\bW_R),\bW_R](\nu)&\hbox{if }\nu \geq 0.
\end{cases}
\end{equation*}
\end{definition}

In the following proposition we provide a sufficient condition on $c^\pm$ and $\ov^\pm$ ensuring that $\RS_{c^-,c^+}$ is well defined on the whole domain $\calW(\ov^-)\times\calW(\ov^+)$.

\begin{proposition}
\label{p:OrbitCulture}
If $c^- > c^+ > 0$ and $\ov^\pm = \ov \in D_V$ are such that
\begin{align}
\label{e:OrbitCulture}
(\rho,w) \in \calW(\ov) &\Longrightarrow  \lambda_w(\rho) \leq 0,
\end{align}
then $D_{\RS} = \calW(\ov^-)\times\calW(\ov^+)$.
\end{proposition}
\begin{proof}
For simplicity of notation, we write $\bW_*$ in place of $\bW_*(\bW_R,w_L)$.
By the hypothesis \eqref{e:OrbitCulture}, we have that
\begin{gather*}
\mathsf{Q}_-(\bW_L) = \mathsf{F}_-(w_L),
\qquad
\mathsf{Q}_+(\bW_R,w_L) = F_{w_L}(\rho_*),
\\
\mathsf{Q}(\bW_L,\bW_R) = \min\left\{
c^- \, \mathsf{F}_-(w_L),
c^+ \, F_{w_L}(\rho_*)
\right\} = c^+ \, F_{w_L}(\rho_*),
\end{gather*}
see \eqref{e:Qm}, \eqref{e:Qp} and \eqref{e:Q}.
By \ref{i:trivcheck} we have that $\check{\bW}(\bW_L,\bW_R) = \bW_*(\bW_R,w_L)$ is well defined and obviously belongs to $\calW(\ov)$.
Case \ref{i:trivhat} is immediate.
We are then left to consider the case $\mathsf{Q}(\bW_L,\bW_R) \neq c^- \, F_{w_L}(\rho_L)$ and to prove that there exists $\bW \in \calW(\ov)$ satisfying \eqref{e:hat}.
Observe that $F_{w_L} \colon [\mathsf{r}_-(w_L),p^{-1}(w_L)] \to [0,\mathsf{F}_-(w_L)]$ is invertible by \eqref{e:OrbitCulture}.
Furthermore, the hypotheses $c^->c^+$ and $\ov^\pm=\ov$ imply that $\frac{c^+}{c^-} \, F_{w_L}(\rho_*) \in [0,\mathsf{F}_-(w_L)]$.
Therefore, it is sufficient to take $\rho = F_{w_L}^{-1} \left( \frac{c^+}{c^-} \, F_{w_L}(\rho_*)\right) \in [\mathsf{r}_-(w_L),p^{-1}(w_L)]$.
This concludes the proof.
\end{proof}

Some remarks on Definition~\ref{d:Haken} are in order.

\begin{remark}
For any $(\bW_L,\bW_R) \in D_{\RS}$, $\bW_\pm \doteq \RS_{c^-,c^+}[\bW_L,\bW_R](0^\pm)$ satisfy the following conditions
\begin{align}
\label{e:Landmarks}
\bW_- = \hat{\bW}(\bW_L,\bW_R),\
\bW_+ = \check{\bW}(\bW_L,\bW_R),\
c^\pm \, F_{w_\pm}(\rho_\pm) = \mathsf{Q}(\bW_L,\bW_R).
\end{align}
\end{remark}

\begin{remark}
We stress that $u(t,x) \doteq \RS_{c^-,c^+}[\bW_L,\bW_R](x/t)$ is a weak solution corresponding to the initial datum $(\bW_L,\bW_R)$.
This simply follows from Proposition~\ref{p:SnarkyPuppy} and \eqref{e:Landmarks}\textsubscript{$3$}.
\end{remark}

\begin{remark}
In the case $c^-\neq c^+$, we have that $u(t,x) \doteq \RS_{c^-,c^+}[\bW_L,\bW_R](x/t)$ is a stationary weak solution if and only if one of the following mutually exclusive cases occurs: either
\begin{equation*}
v(\bW_L) = v(\bW_R) = 0,
\end{equation*}
or
\begin{align*}
&w_L = w_R,&
&\rho_L \leq \mathsf{r}_-(w_L),&
&c^- \, F_{w_L}(\rho_L) = c^+ \, F_{w_R}(\rho_R),
\end{align*}
or
\begin{align*}
&w_L = w_R,&
&\rho_L > \mathsf{r}_-(w_L),&
&\rho_R > \mathsf{r}_+(w_R),&
&c^- \, F_{w_L}(\rho_L) = c^+ \, F_{w_R}(\rho_R).
\end{align*}

\end{remark}

\begin{remark}
Definition~\ref{d:Haken} resembles \cite[Definition~2.1]{Colombo-Goatin_constraint} for a scalar conservation law with a point constraint on the flow at $x=0$, which in our case becomes $c^\pm \, F_{w_\pm}(\rho_\pm) \leq \mathsf{Q}(\bW_L,\bW_R)$. 
The key difference is that we do not check whether the classical solution satisfies the constraint. Instead, we emphasize selecting, among the self-similar weak solutions that meet the constraint, the one that optimizes the flow.
\end{remark}

\begin{remark}
For any $(\bW_L,\bW_R)$ such that $w_L\neq w_R$, the rightmost wave of $\RS_{c^-,c^+}[\bW_L,\bW_R]$ is a $2$-contact discontinuity propagating with speed $c^+ \, V(v(\bW_R))$.
We stress that $2$-contact discontinuities always separate the two families of vehicles, characterized by the Lagrangian markers $w_L$ and $w_R$.
\end{remark}

\begin{remark}
$\RS_{c^-,c^+}[\bW_L,\bW_R]$ expressed in the Riemann invariant coordinates \eqref{e:vw} may have total variation greater than that of the initial datum, that is $|v(\bW_L)-v(\bW_R)|+|w_L-w_R|$; moreover the maximum principle holds for the $w$-coordinate, but may fail for the $v$-coordinate.
Indeed, if $\bW_L = \bW_R$, $v(\bW_L) = v(\bW_R) \neq 0$ and $c^- \neq c^+$, then $\RS_{c^-,c^+}[\bW_L, \bW_L]$ is not constant, see Section~\ref{s:4} for explicit examples.
\end{remark}

\begin{remark}
We have a stationary $2$-contact discontinuity at $\nu=0$ if and only if $v(\bW_R)=0$.
If this condition is not satisfied, the presence of a stationary discontinuity at $\nu=0$ is instead due to the discontinuity of $c$.
In this case, we refer to such a discontinuity as a $c$-wave.
\end{remark}

We conclude this section by studying the main properties of the Riemann solver $\RS_{c^-,c^+}$. Specifically, we consider the coherence and the $\Lloc{1}$-continuity with respect to the initial datum. 
We recall that, as in \cite[Definition~2.2]{CorliFigielFutaRosini}, the Riemann solver is coherent at $(\bW_L, \bW_R)$ if
\begin{align}
\label{e:PorcupineTree}
&\bW_\pm = \RS_{c^-,c^+}[\bW_L, \bW_R](0^\pm)&
&\Longrightarrow&
\RS_{c^-,c^+}[\bW_-, \bW_+](0^\pm) = \bW_\pm.
\end{align}
Observe that coherence can be thought of as a stability property: the ordered pair of traces is, in a sense, a fixed point for $\RS_{c^-,c^+}$.

\begin{proposition}
The Riemann solver $\RS_{c^-,c^+} \colon D_{\RS} \to \BV(\R;\calW)$ is coherent and $\Lloc{1}$-continuous with respect to the initial datum.
\end{proposition}
\noindent
\begin{proof}
We first prove that the Riemann solver $\RS_{c^-,c^+}$ is coherent in $D_{\RS}$.
To do so, it is sufficient to show that for any $(\bW_L,\bW_R) \in D_{\RS}$ we have
\begin{equation}
\label{e:CoheedandCambria}
\mathsf{Q}(\bW_-,\bW_+) = \mathsf{Q}(\bW_L,\bW_R)
\end{equation}
where $\bW_\pm = \RS_{c^-,c^+}[\bW_L,\bW_R](0^\pm)$.
Indeed, by \eqref{e:Landmarks}\textsubscript{$3$} and \eqref{e:CoheedandCambria} we have
\[c^\pm \, F_{w_\pm}\left(\rho_\pm\right) = \mathsf{Q}(\bW_-,\bW_+).\]
By \ref{i:trivhat}, \ref{i:trivcheck} and \eqref{e:Landmarks}\textsubscript{$1,2$} we have then $\RS_{c^-,c^+}[\bW_-,\bW_+](0^\pm) = \bW_\pm$ and therefore $\RS_{c^-,c^+}$ is coherent at $(\bW_L, \bW_R)$, see \eqref{e:PorcupineTree}.
We are therefore left to prove \eqref{e:CoheedandCambria}.
Suppose, by contradiction, that there exists $(\bW_L, \bW_R) \in D_{\RS}$ such that \eqref{e:CoheedandCambria} is not satisfied, namely
\begin{equation}
\label{e:Blackfield}
\mathsf{Q}(\bW_-,\bW_+) \neq \mathsf{Q}(\bW_L,\bW_R).
\end{equation}
By \ref{i:trivhat}, \eqref{e:hat}\textsubscript{$2$}, \eqref{e:w*}\textsubscript{$2$}, \ref{i:trivcheck} and \eqref{e:check}\textsubscript{$2$} we have $\hat{w}(\bW_L,\bW_R) = \check{w}(\bW_L,\bW_R) = w_L$, hence by \eqref{e:Landmarks}\textsubscript{$1,2$}
\begin{equation}
\label{e:Klone}    
w_\pm = w_L.
\end{equation}
This implies that $\bW_*(\bW_+,w_-) = \bW_+$, see \eqref{e:w*}.
By \eqref{e:Qm}, \eqref{e:Qp}, \eqref{e:Q}, \eqref{e:Landmarks}\textsubscript{$3$} and \eqref{e:Blackfield} we have that $\mathsf{Q}(\bW_-,\bW_+) > 
\mathsf{Q}(\bW_L,\bW_R) =
c^\pm \, F_{w_\pm}(\rho_\pm)$.
As a consequence, by \ref{i:hat}, \eqref{e:hat}\textsubscript{$1$}, \ref{i:check}, \eqref{e:check}\textsubscript{$1$} and \eqref{e:Klone} we have
\begin{align*}
&\rho_- > \mathsf{r}_-(w_L),&
&\rho_+ = \rho_*(\bW_+,w_-) < \mathsf{r}_+(w_L).
\end{align*}
The above conditions imply that
\begin{align*}
\hbox{either } \quad F_{w_L}(\rho_L) > F_{w_-}(\rho_-) \quad \hbox{ or } \quad
&\begin{cases}
F_{w_L}(\rho_L) \leq F_{w_-}(\rho_-),
\\
\rho_L > \mathsf{r}_-(w_L), 
\end{cases}
\\
\hbox{either } \quad F_{w_*}(\rho_*) > F_{w_+}(\rho_+) \quad \hbox{ or } \quad
&\begin{cases}
F_{w_*}(\rho_*) \leq F_{w_+}(\rho_+),
\\
\rho_* < \mathsf{r}_+(w_L).
\end{cases}
\end{align*}
In any case, we get $\mathsf{Q}_-(\bW_L) > F_{w_L}(\rho_-)$ and $\mathsf{Q}_+(\bW_R,w_L) > F_{w_L}(\rho_+)$, hence $\mathsf{Q}(\bW_L,\bW_R) > c^\pm \, F_{w_\pm}(\rho_\pm)$.
This gives a contradiction and therefore \eqref{e:CoheedandCambria} holds.
As a consequence, $\RS_{c^-,c^+}$ is coherent.

We now prove that the Riemann solver $\RS_{c^-,c^+}$ is $\Lloc1$-continuous with respect to the initial data in $D_{\RS}$.
The case $c^- = c^+$ is trivial and is therefore omitted.
Observe first that  $\mathsf{Q}$ is continuous because $\mathsf{Q}_-$ and $\mathsf{Q}_+$ are so.
By the continuity of the speed of propagation of the waves, it is sufficient to consider initial data for which the traces at $\nu=0$ are discontinuous.



More precisely, assume that $(\bW_L^\varepsilon, \bW_R^\varepsilon)$ converges to $(\bW_L,\bW_R)$ (as $\varepsilon$ tends to zero), but the traces $\bW_\pm^\varepsilon \doteq \RS_{c^-,c^+}[\bW_L^\varepsilon, \bW_R^\varepsilon](0^\pm)$ do not converge to $\bW_\pm \doteq \RS_{c^-,c^+}[\bW_L, \bW_R](0^\pm)$.
By \eqref{e:Landmarks}\textsubscript{$1,3$} and Definition~\ref{d:hat}, we have to consider the following cases:
\begin{itemize}[leftmargin=*]
\item
Assume that $\rho_L \leq \mathsf{r}_-(w_L)$ and $\mathsf{Q}(\bW_L,\bW_R) = c^- \, F_{w_L}(\rho_L)$, so that $\bW_- = \bW_L$.
Consider now a perturbed initial datum $(\bW_L^\varepsilon, \bW_R^\varepsilon)$ such that $\rho_L^\varepsilon \leq \mathsf{r}_-(w_L^\varepsilon)$, $\mathsf{Q}(\bW_L^\varepsilon,\bW_R^\varepsilon) < c^- \, F_{w_L^\varepsilon}(\rho_L^\varepsilon)$, so that $\bW_-^\varepsilon \neq \bW_L^\varepsilon$.
Observe that $\RS_-[\bW_L,\bW_-]$ is constant, whereas $\RS_-[\bW_L^\varepsilon,\bW_-^\varepsilon]$ has just one wave, a $1$-shock, with speed of propagation that converges to zero as $\varepsilon \to 0$.
This ensures the $\Lloc1$-convergence of $\RS_{c^-,c^+}[\bW_L^\varepsilon,\bW_R^\varepsilon]$ to $\RS_{c^-,c^+}[\bW_L,\bW_R]$.
\item 
Assume that $\rho_* \geq \mathsf{r}_+(w_L)$ and $\mathsf{Q}(\bW_L,\bW_R) = c^+ \, F_{w_L}(\rho_*)$, so that $\bW_+ = \bW_*$.
Consider now a perturbed initial datum $(\bW_L^\varepsilon, \bW_R^\varepsilon)$ such that $\rho_*^\varepsilon \geq \mathsf{r}_+(w_L^\varepsilon)$, $\mathsf{Q}(\bW_L^\varepsilon,\bW_R^\varepsilon) < c^+ \, F_{w_L^\varepsilon}(\rho_*^\varepsilon)$, so that $\bW_+^\varepsilon \neq \bW_*^\varepsilon$.
Observe that $\RS_+[\bW_+,\bW_*]$ is constant, whereas $\RS_+[\bW_+^\varepsilon,\bW_*^\varepsilon]$ has just one wave, a $1$-shock, with speed of propagation that converges to zero as $\varepsilon \to 0$.
This ensures the $\Lloc1$-convergence of $\RS_{c^-,c^+}[\bW_L^\varepsilon,\bW_R^\varepsilon]$ to $\RS_{c^-,c^+}[\bW_L,\bW_R]$.
\end{itemize}
This concludes the proof.
\end{proof}

\section{Case studies for Riemann problems}
\label{s:4}

In this section, we present the construction of weak solutions to some Riemann problems \eqref{e:Riemann}.
The initial data are chosen to illustrate the main properties of the Riemann solver $\RS_{c^-,c^+}$ introduced in Definition~\ref{d:Haken}.
The values $\ov^\pm$ are not fixed explicitly; rather, we assume that they are sufficiently large to ensure the validity of the construction below.
We consider the following choices for the functions $p$ and $V$, together with their respective domains of definition:
\begin{align*}
&p(\rho) = \ln(\rho), \ D_p=(0,1], \ I_p=(-\infty,0],&
&V(v) = v, \ D_V = [0,\infty),
\\
&p(\rho) = \rho^\gamma,\ D_p=[0,\infty), \ I_p=[0,\infty),&
&V(v) = \sinh(v), \ D_V = [0,\infty),
\\
&p(\rho) = \left(\frac{\rho}{1-\rho}\right)^{\gamma},\ D_p=[0,1), \ I_p=[0,\infty),
\end{align*}
with $\gamma > 0$.
Clearly, the functions $p$ and $V$ defined above satisfy conditions \eqref{e:p} and \eqref{e:V}, respectively.
Moreover, condition \eqref{e:pV1} reduces to the following inequalities: for any $\rho \in D_p\setminus\{0\}$ and $v \in D_V$ we have
\begin{align*}
p(\rho) = \ln(\rho): \ & \frac{V''(v)}{V'(v)} < 1,&
p(\rho) = \rho^\gamma: \ & \frac{V''(v)}{V'(v)} < \frac{\gamma+1}{\gamma \, \rho^\gamma},
\\
p(\rho) = \left(\frac{\rho}{1-\rho}\right)^{\!\!\gamma}: \ & \frac{V''(v)}{V'(v)} < \left(1+\frac{1}{\!\!\gamma}\right)  \left(\frac{1}{\rho}-1\right)^{\gamma},
\\
V(v) = \sinh(v): \ & \tanh(v) < \frac{2p'(\rho)+\rho\,p''(\rho)}{\rho \, p'(\rho)^2},&
V(v) = v: \ & 0 < \frac{2p'(\rho)+\rho\,p''(\rho)}{\rho \, p'(\rho)^2}.
\end{align*}
We recall that the case with $p(\rho) = \ln(\rho)$ is analyzed in detail in \cite{ChiarelloFagioliRosini1}.

We represent the construction of the solutions to Riemann problems in the $(\rho,f)$-plane.
Observe that $\bW = (\rho,w)$ is represented in the $(\rho,f)$-plane by the two points $(\rho,f^-) = \left( \rho , c^- \, F_w(\rho) \right)$ and $(\rho,f^+) = \left( \rho , c^+ \, F_w(\rho) \right)$.
This motivates the introduction of the following notation: for a given $\bW_L \in \calW(\ov^-)$, take
\begin{align*}
&\bF_L^- \doteq \left( \rho_L , c^- \, F_{w_L}(\rho_L) \right)&
&\bF_L^+ \doteq \left( \rho_L , c^+ \, F_{w_L}(\rho_L) \right).
\end{align*}
We use analogous notation for $\bF_R^\pm$, $\bF_*^\pm$, and so on.

We start by considering a constant initial datum $\bW_o$, but with $c^- \neq c^+$, to show the impact of $c^\pm$ on the weak solution $u(t,x) = \RS_{c^-,c^+}[\bW_o, \bW_o](x/t)$, that the maximum principle is not satisfied and that the total variation of the weak solution $u$ expressed in the Riemann invariant coordinates \eqref{e:vw} may increase instantaneously.

\begin{example}
\label{ex:1}
Consider the Riemann problem~\eqref{e:Riemann} with
\begin{align}
\label{e:iniconst}
&c^-=2,&
&c^+=3,&
&\bW_L = \bW_R = \bW_o.
\end{align}
We detail below two possible choices of $(\bW_o,p,V)$.
\begin{figure}[!ht]
\centering
\resizebox{.99\linewidth}{!}{
\begin{tikzpicture}[every node/.style={anchor=south west,inner sep=0pt},x=1mm, y=1mm]

\node at (0,0) {\includegraphics[width=35mm]{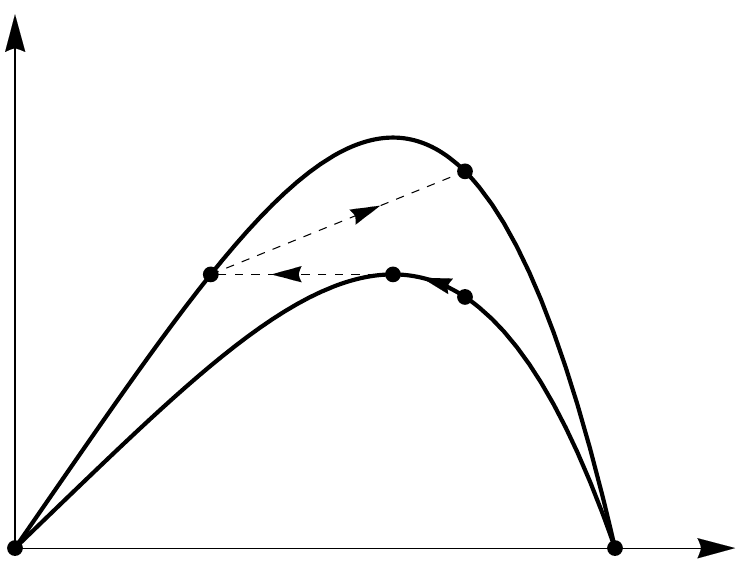}};
\node[below] at (35,0) {\strut $\rho$};
\node[left] at (-1,26.25) {\strut $c\,F_w$};
\draw[latex-] (22.5,13) --++ (5,3) node[right] {\strut $\bF_o^-$};
\node at (22,18) {\strut $\bF_o^+$};
\node at (16.5,8) {\strut $\hat{\bF}^-$};
\node at (6,12.5) {\strut $\check{\bF}^+$};

\begin{scope}[shift={(37,0)}]
\node at (0,0) {\includegraphics[width=35mm]{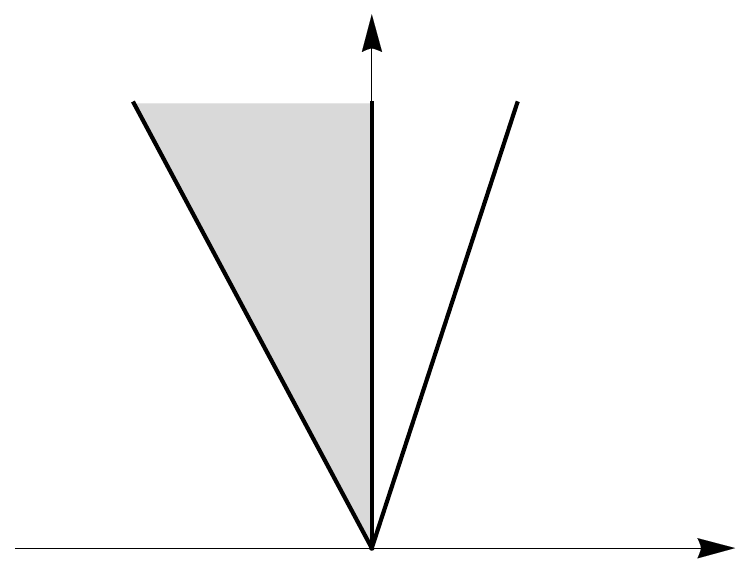}};
\node[below] at (35,0) {\strut $x$};
\node[left] at (16,26.25) {\strut $t$};
\node at (5,7) {\strut $\bW_o$};
\node at (25,4) {\strut $\bW_o$};
\draw[latex-] (17,22) --++ (-5,2) node[left] {\strut $\hat{\bW}$};
\node at (18,16) {\strut $\check{\bW}$};
\end{scope}

\begin{scope}[shift={(74,0)}]
\node at (0,0) {\includegraphics[width=35mm]{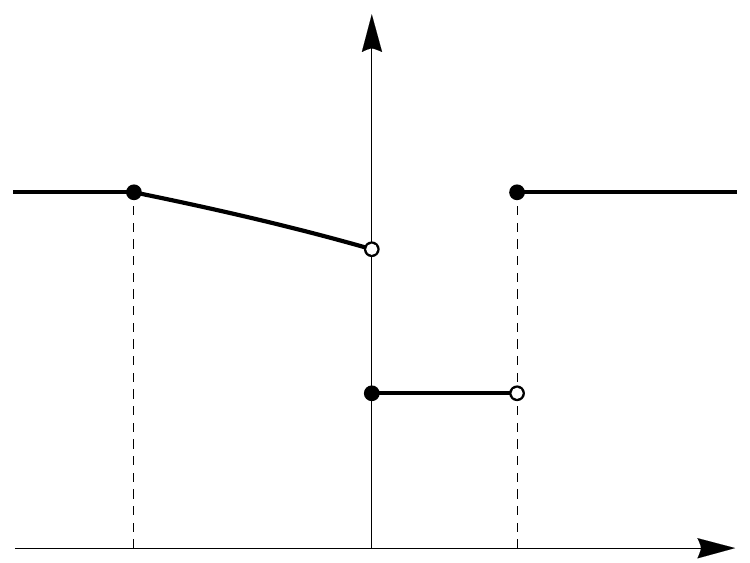}};
\node[below] at (35,0) {\strut $x$};
\node[below] at (8,0) {\strut $x_1$};
\node[below] at (25,0) {\strut $x_2$};
\node[left] at (16,26.25) {\strut $\rho$};
\node at (2,18) {\strut $\rho_o$};
\node at (28,18) {\strut $\rho_o$};
\node at (18.5,13) {\strut $\hat{\rho}$};
\node at (20,8) {\strut $\check{\rho}$};
\end{scope}

\begin{scope}[shift={(111,0)}]
\node at (0,0) {\includegraphics[width=35mm]{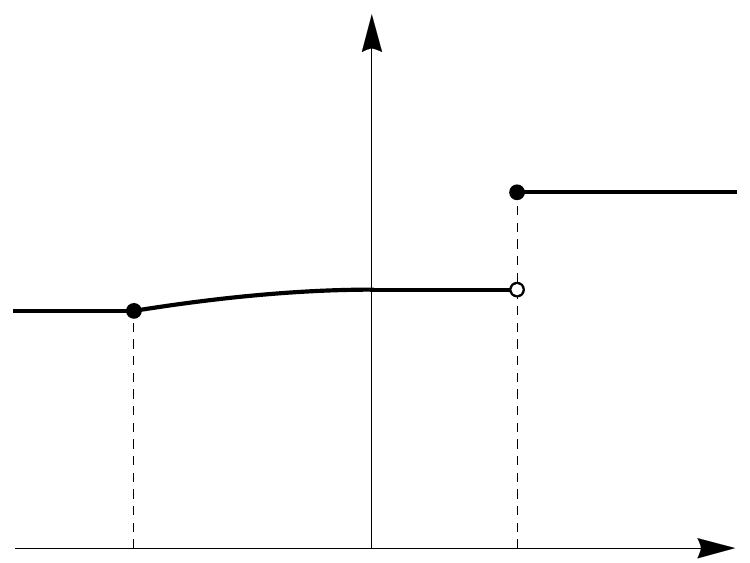}};
\node[below] at (35,0) {\strut $x$};
\node[below] at (8,0) {\strut $x_1$};
\node[below] at (25,0) {\strut $x_2$};
\node[left] at (16,26.25) {\strut $c\,F_w$};
\node at (19,12.5) {\strut $\mathsf{Q}$};
\node[right] at (1,14.5) {\strut $F_{w_o}^-(\rho_o)$};
\node at (22,18) {\strut $F_{w_o}^+(\rho_o)$};
\end{scope}

\begin{scope}[shift={(148,0)}]
\node at (0,0) {\includegraphics[width=35mm]{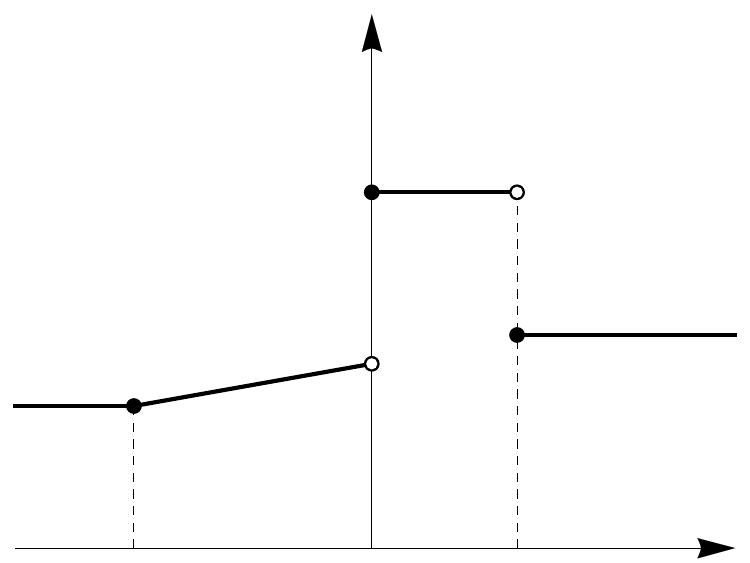}};
\node[below] at (35,0) {\strut $x$};
\node[below] at (8,0) {\strut $x_1$};
\node[below] at (25,0) {\strut $x_2$};
\node[left] at (16,26.25) {\strut $c\,V$};
\node at (1,7) {\strut $c^-V_o$};
\node at (26,10.5) {\strut $c^+V_o$};
\draw[latex-] (17,10) --++ (-6,4) node[left, inner sep=2pt] {\strut $c^-\hat{V}$};
\node at (18,17) {\strut $c^+\check{V}$};
\end{scope}

\end{tikzpicture}
}
\caption{Representation of the weak solution $u$ corresponding to \eqref{e:iniconst} with $\bW_o = (0.75,1)$, $p(\rho) = \rho^3$ and $V(v) = v$.
The profiles of $\rho$, $c\,F_w$ and $c\,V$ are computed at time $t=1$.
Above we take $V_o = V(v(\bW_o))$, $\hat{V} = V(v(\hat{\bW}(\bW_o,\bW_o)))$, $\check{V} = V(v(\check{\bW}(\bW_o,\bW_o)))$, $\mathsf{Q}=\mathsf{Q}(\bW_L,\bW_R)$ and $F_w^\pm = c^\pm \, F_w$.
Moreover we let $x_1 = c^- \, \lambda_{w_o}(\rho_o) \approx -1.375$ and $x_2 = c^+ \, s(\check{\bW}(\bW_o,\bW_o),\bW_o) \approx 0.840$.
} 
\label{f:sol-pow-id}
\end{figure}
\begin{enumerate}[label={{\rm\bf(\Roman*)}}, leftmargin=*]
\item 
\label{i:sol-log-sinh}
In \figurename~\ref{f:sol-pow-id} we show the construction of the weak solution $u$ in the case $\bW_o = (0.75,1)$, $p(\rho) = \rho^3$ and $V(v) = v$.
In this case $\rho_L > \mathsf{r}_-(w_L)$, $\bW_*(\bW_R,w_L) = \bW_R$ and $\rho_*(\bW_R,w_L) > \mathsf{r}_+(w_L)$, therefore by definition, see \eqref{e:Qm}, \eqref{e:Qp} and \eqref{e:Q}, we have
\begin{align*}
&\mathsf{Q}_-(\bW_L) = \mathsf{F}_-(w_o),&
&\mathsf{Q}_+(\bW_R,w_L) = F_{w_o}(\rho_o),&
&\mathsf{Q}(\bW_L,\bW_R) = c^- \, \mathsf{F}_-(w_o).
\end{align*}
By \ref{i:hat} and \ref{i:check} we have then $\hat{\bW}(\bW_L,\bW_R) = (\mathsf{r}_-(w_o),w_o)$ and $\check{\bW}(\bW_L,\bW_R) = (F_{w_o}^{-1}(\frac{c^-}{c^+} \, \mathsf{F}_-(w_o)) , w_o)$.
Therefore, $u$ has a $1$-rarefaction in the cone $c^- \, \lambda_{w_o}(\rho_o) \leq x/t \leq c^- \, \lambda_{w_o}(\mathsf{r}_-(w_o)) = 0$, a stationary $c$-wave at $x=0$ and a $1$-shock at $x = c^+ \, s(\check{\bW}(\bW_o,\bW_o),\bW_o) \, t$, with $s$ given in \eqref{e:s}.
\begin{figure}[!ht]
\centering
\resizebox{.25\linewidth}{!}{
\begin{tikzpicture}[every node/.style={anchor=south west,inner sep=0pt},x=1mm, y=1mm]

\node at (0,0) {\includegraphics[width=35mm]{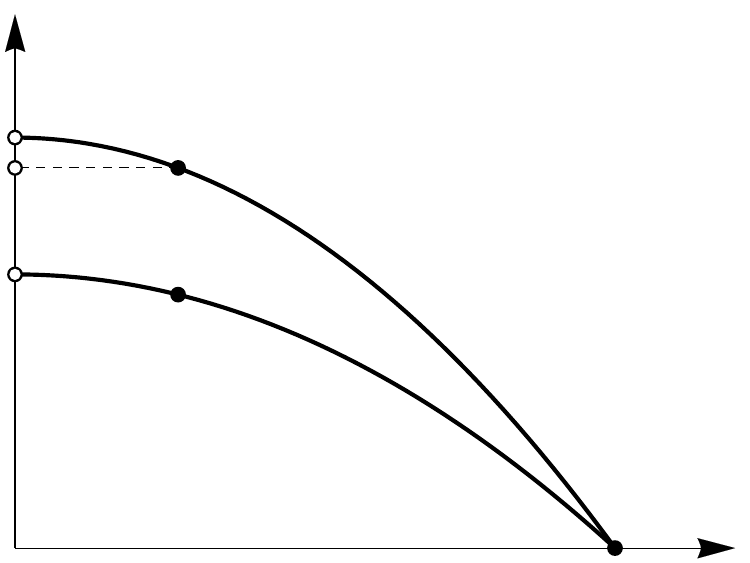}};
\node[below] at (35,0) {\strut $\rho$};
\node[left] at (-1,26.25) {\strut $c\,F_w$};
\node at (9,12) {\strut $\bF_o^-$};
\node at (8,18) {\strut $\bF_o^+$};

\end{tikzpicture}
}
\caption{
In the case corresponding to \eqref{e:iniconst} with $\bW_o = (0.1,-1)$, $p(\rho) = \ln(\rho)$ and $V(v) = \sinh(v)$, the Riemann solver given in Definition~\ref{d:Haken} is not well defined!
} 
\label{f:sol-log-sinh}
\end{figure}
\item
In \figurename~\ref{f:sol-log-sinh} we consider the case $\bW_o = (0.1,-1)$, $p(\rho)=\ln(\rho)$ and $V(v)=\sinh(v)$.
In this case $\rho_L > \mathsf{r}_-(w_L)$, $\bW_*(\bW_R,w_L) = \bW_R$ and $\rho_*(\bW_R,w_L) > \mathsf{r}_+(w_L)$, therefore by definition, see \eqref{e:Qm}, \eqref{e:Qp} and \eqref{e:Q}, we have
\begin{align*}
&\mathsf{Q}_-(\bW_L) = \mathsf{F}_-(w_o),&
&\mathsf{Q}_+(\bW_R,w_L) = F_{w_o}(\rho_o),&
&\mathsf{Q}(\bW_L,\bW_R) = c^- \, \mathsf{F}_-(w_o).
\end{align*}
On one hand, by \ref{i:hat} we have that $\hat{\bW}(\bW_L,\bW_R) = (\mathsf{r}_-(w_o),w_o)$.
On the other hand, $\mathsf{Q}(\bW_L,\bW_R) \neq c^+ \, F_{w_o}(\rho_o)$ and condition \eqref{e:check} has no solution.
Hence $\check{\bW}(\bW_L,\bW_R)$ is not well defined.
As a consequence $(\bW_o,\bW_o)$ is not in $D_{\RS}$ and $\RS_{c^-,c^+}[\bW_o,\bW_o]$ is not well defined.
\end{enumerate}
\end{example}

\begin{example}
\label{ex:2}
We consider the two choices of $(p,V)$ as in Example~\ref{ex:1}, and the data set given in \eqref{e:iniconst}, but with the values of $c^-$ and $c^+$ inverted, that is
\begin{align}
\label{e:iniconst-inv}
&c^- = 3,&
&c^+ = 2,&
&\bW_L = \bW_R = \bW_o.
\end{align}
This choice leads to a qualitatively different structure of the weak solution to the Riemann problem.
\begin{figure}[!ht]
\centering
\resizebox{.99\linewidth}{!}{
\begin{tikzpicture}[every node/.style={anchor=south west,inner sep=0pt},x=1mm, y=1mm]

\node at (0,0) {\includegraphics[width=35mm]{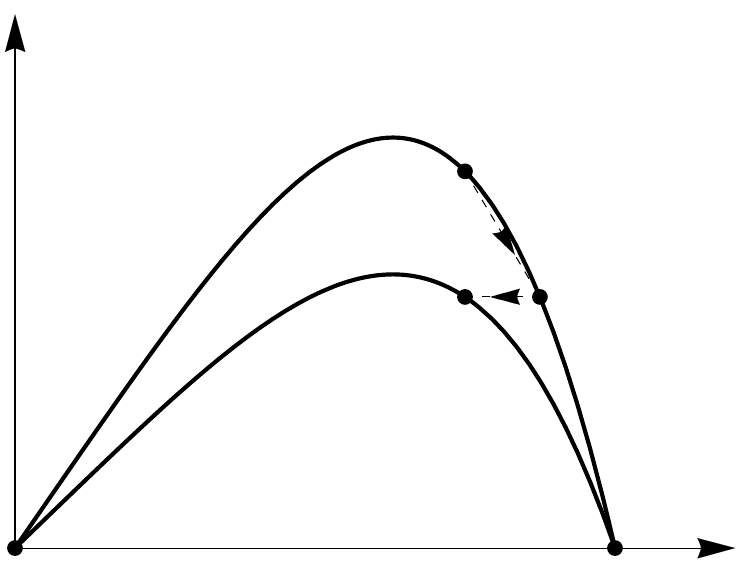}};
\node[below] at (35,0) {\strut $\rho$};
\node[left] at (-1,26.25) {\strut $c\,F_w$};
\node at (17.5,8) {\strut $\bF_o^+$};
\node at (22,18) {\strut $\bF_o^-$};
\node at (26.5,10) {\strut $\hat{\bF}^-$};

\begin{scope}[shift={(37,0)}]
\node at (0,0) {\includegraphics[width=35mm]{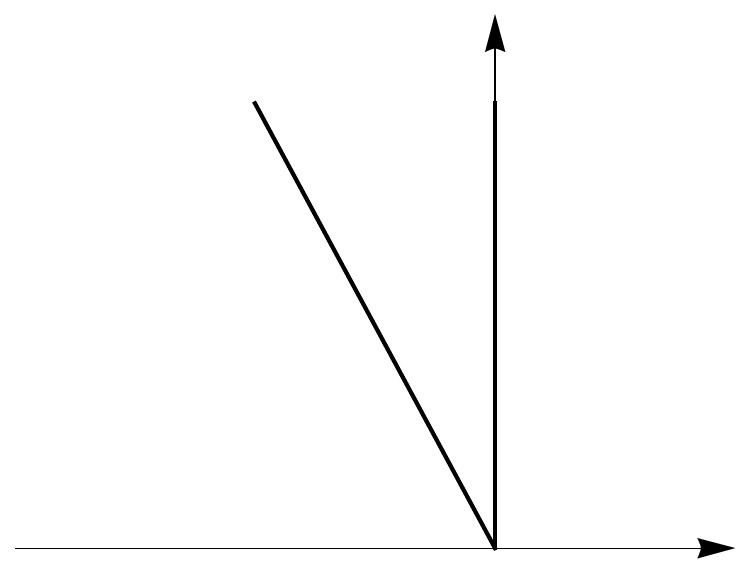}};
\node[below] at (35,0) {\strut $x$};
\node[left] at (21,26.25) {\strut $t$};
\node at (5,7) {\strut $\bW_o$};
\node at (26,10) {\strut $\bW_o$};
\node at (16,16) {\strut $\hat{\bW}$};
\end{scope}

\begin{scope}[shift={(74,0)}]
\node at (0,0) {\includegraphics[width=35mm]{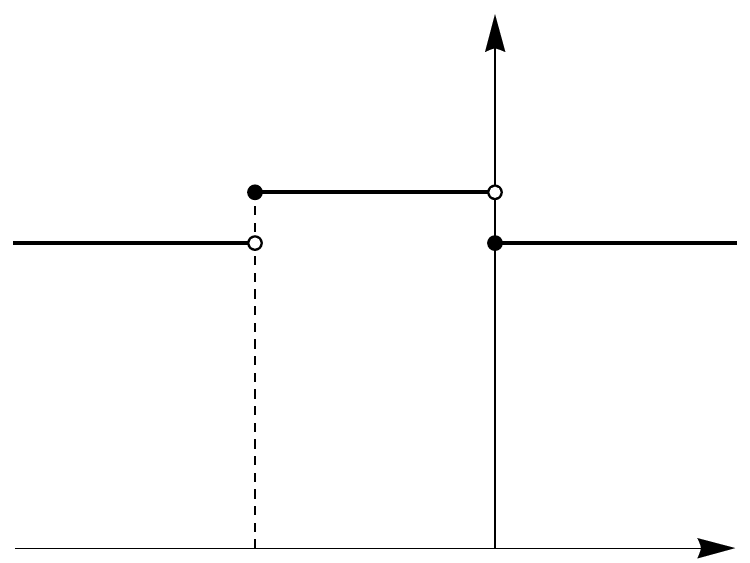}};
\node[below] at (35,0) {\strut $x$};
\node[below] at (13,0) {\strut $x_1$};
\node[left] at (21,26.25) {\strut $\rho$};
\node at (3,15) {\strut $\rho_o$};
\node at (28,15) {\strut $\rho_o$};
\node at (17,17) {\strut $\hat{\rho}$};
\end{scope}

\begin{scope}[shift={(111,0)}]
\node at (0,0) {\includegraphics[width=35mm]{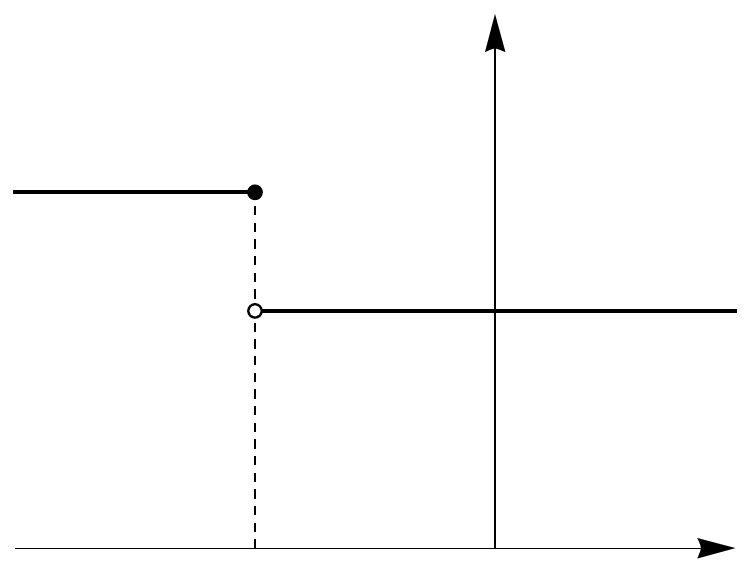}};
\node[below] at (35,0) {\strut $x$};
\node[below] at (13,0) {\strut $x_1$};
\node[left] at (21,26.25) {\strut $c\,F_w$};
\node[right] at (1,20) {\strut $F_{w_o}^-(\rho_o)$};
\node at (23,12) {\strut $F_{w_o}^+(\rho_o)$};
\end{scope}

\begin{scope}[shift={(148,0)}]
\node at (0,0) {\includegraphics[width=35mm]{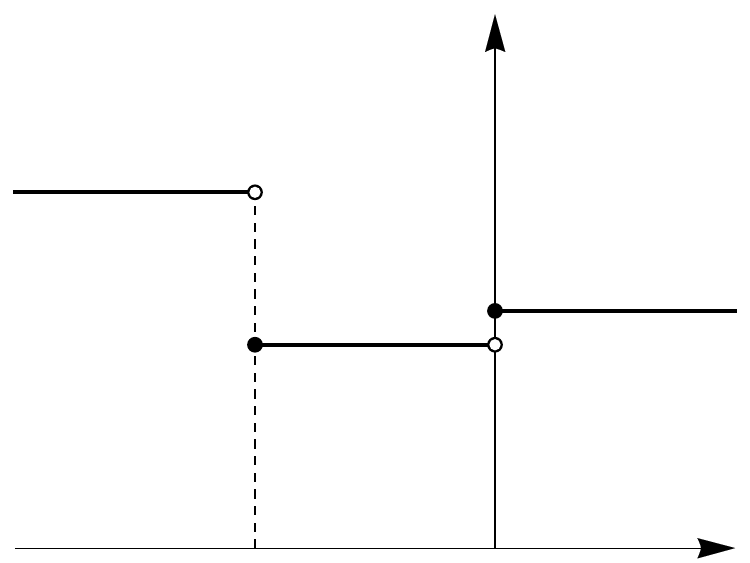}};
\node[below] at (35,0) {\strut $x$};
\node[below] at (13,0) {\strut $x_1$};
\node[left] at (23,26.25) {\strut $c\,V$};
\node at (1,17) {\strut $c^-V_o$};
\node at (26,12) {\strut $c^+V_o$};
\node at (14,10) {\strut $c^-\hat{V}$};
\end{scope}

\end{tikzpicture}
}
\caption{Representation of the weak solution $u$ corresponding to \eqref{e:iniconst-inv} with $\bW_o = (0.75,1)$, $p(\rho) = \rho^3$ and $V(v) = v$.
The profiles of $\rho$, $c\,F_w$ and $c\,V$ are computed at time $t=1$.
Above we take $V_o = V(v(\bW_o))$, $\hat{V} = V(v(\hat{\bW}(\bW_o,\bW_o)))$ and $F_w^\pm = c^\pm \, F_w$.
Moreover we let $x_1 = c^- \, s(\bW_o,\hat{\bW}(\bW_o,\bW_o)) \approx -3.473$.}
\label{f:sol-pow-id-inv}
\end{figure}
\begin{enumerate}[label={{\rm\bf(\Roman*)}}, leftmargin=*]
\item 
\label{i:sol-pow-id-inv}
In \figurename~\ref{f:sol-pow-id-inv} we show the construction of the weak solution $u$ in the case $\bW_o = (0.75,1)$, $p(\rho) = \rho^3$ and $V(v) = v$.
In this case $\rho_L > \mathsf{r}_-(w_L)$, $\bW_*(\bW_R,w_L) = \bW_R$ and $\rho_*(\bW_R,w_L) > \mathsf{r}_+(w_L)$, therefore by definition, see \eqref{e:Qm}, \eqref{e:Qp} and \eqref{e:Q}, we have $\mathsf{Q}_-(\bW_L) = \mathsf{F}_-(w_o)$, $\mathsf{Q}_+(\bW_R,w_L) = F_{w_o}(\rho_o)$ and $\mathsf{Q}(\bW_L,\bW_R) = c^+ \, F_{w_o}(\rho_o)$.
By \ref{i:hat} and \ref{i:trivcheck} we have then 
\begin{align*}
&\hat{\bW}(\bW_L,\bW_R) = \left(F_{w_o}^{-1}\left(\frac{c^+}{c^-} \, F_{w_o}(\rho_o)\right) , w_o\right),&
&\check{\bW}(\bW_L,\bW_R) = \bW_o.    
\end{align*}
Therefore, $u$ has a $1$-shock at $x = c^- \, s(\bW_o,\hat{\bW}(\bW_o,\bW_o)) \, t$ and a stationary $c$-wave at $x=0$.
\begin{figure}[!ht]
\centering
\resizebox{.99\linewidth}{!}{
\begin{tikzpicture}[every node/.style={anchor=south west,inner sep=0pt},x=1mm, y=1mm]

\node at (0,0) {\includegraphics[width=35mm]{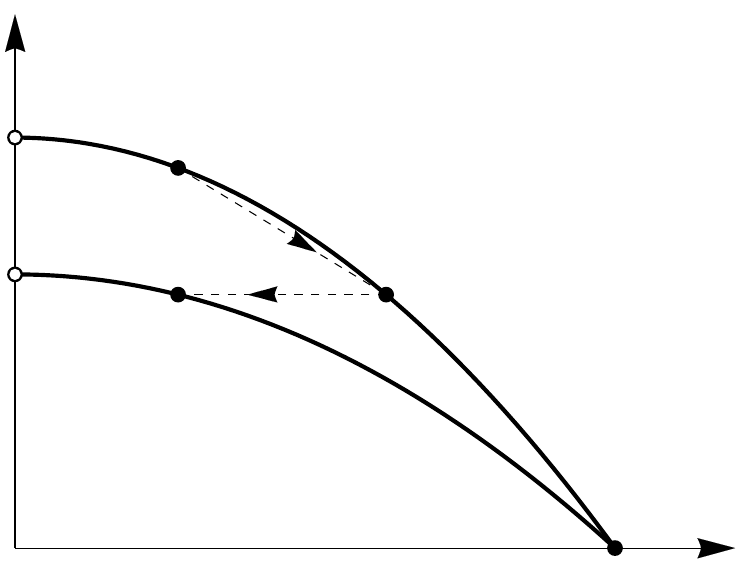}};
\node[below] at (35,0) {\strut $\rho$};
\node[left] at (-1,26.25) {\strut $c\,F_w$};
\node at (7,18.5) {\strut $\bF_o^-$};
\node at (6,13) {\strut $\bF_o^+$};
\node at (19,11.5) {\strut $\hat{\bF}^-$};

\begin{scope}[shift={(37,0)}]
\node at (0,0) {\includegraphics[width=35mm]{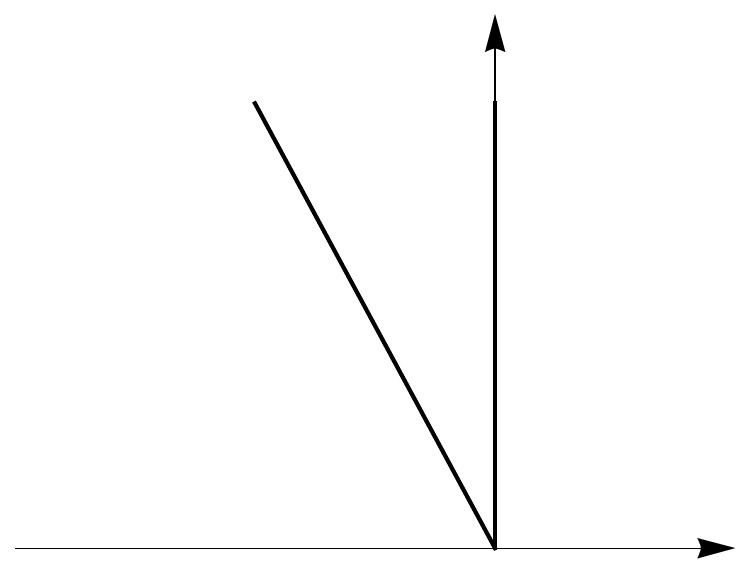}};
\node[below] at (35,0) {\strut $x$};
\node[left] at (21,26.25) {\strut $t$};
\node at (5,7) {\strut $\bW_o$};
\node at (25,10) {\strut $\bW_o$};
\node at (17,15) {\strut $\hat{\bW}$};
\end{scope}

\begin{scope}[shift={(74,0)}]
\node at (0,0) {\includegraphics[width=35mm]{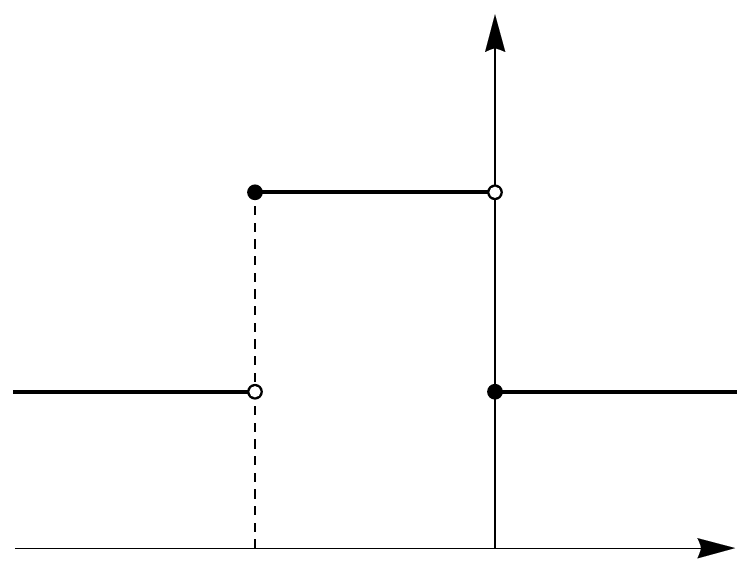}};
\node[below] at (35,0) {\strut $x$};
\node[below] at (13,0) {\strut $x_1$};
\node[left] at (21,26.25) {\strut $\rho$};
\node at (3,8) {\strut $\rho_o$};
\node at (28,8) {\strut $\rho_o$};
\node at (17,17.5) {\strut $\hat{\rho}$};
\end{scope}

\begin{scope}[shift={(111,0)}]
\node at (0,0) {\includegraphics[width=35mm]{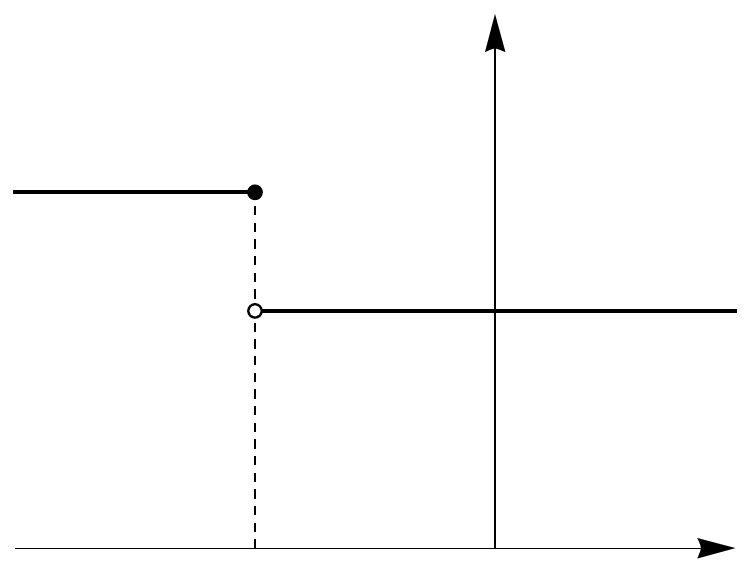}};
\node[below] at (35,0) {\strut $x$};
\node[below] at (13,0) {\strut $x_1$};
\node[left] at (21,26.25) {\strut $c\,F_w$};
\node[right] at (1,20) {\strut $F_{w_o}^-(\rho_o)$};
\node[left] at (29.5,14.5) {\strut $F_{w_o}^+(\rho_o)$};
\end{scope}

\begin{scope}[shift={(148,0)}]
\node at (0,0) {\includegraphics[width=35mm]{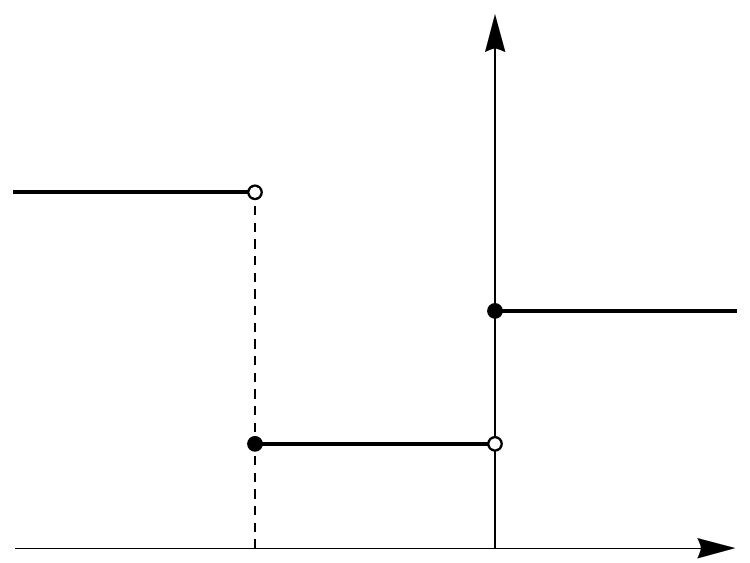}};
\node[below] at (35,0) {\strut $x$};
\node[below] at (13,0) {\strut $x_1$};
\node[left] at (21,26.25) {\strut $c\,V$};
\node at (1,17.5) {\strut $c^-V_o$};
\node at (26,12) {\strut $c^+V_o$};
\node at (14,5.5) {\strut $c^-\hat{V}$};
\end{scope}

\end{tikzpicture}
}
\caption{Representation of the weak solution $u$ corresponding to \eqref{e:iniconst-inv} with $\bW_o = (0.1,-1)$, $p(\rho) = \ln(\rho)$ and $V(v) = \sinh(v)$.
The profiles of $\rho$, $c\,F_w$ and $c\,V$ are computed at time $t=1$.
Above we use the same notation as in \figurename~\ref{f:sol-pow-id-inv} and $x_1 \approx -1.336$.}
\label{f:sol-log-sinh-inv}
\end{figure}
\item
In \figurename~\ref{f:sol-log-sinh-inv} we consider the case $\bW_o = (0.1,-1)$, $p(\rho)=\ln(\rho)$ and $V(v)=\sinh(v)$.
The construction is similar to the previous case and is omitted.
\end{enumerate}
\end{example}

In the next example, we consider a Riemann initial datum away from the vacuum that leads to a weak solution that develops a vacuum state immediately for $t>0$.
\begin{figure}[!ht]
\centering
\resizebox{.99\linewidth}{!}{
\begin{tikzpicture}[every node/.style={anchor=south west,inner sep=0pt},x=1mm, y=1mm]

\node at (0,0) {\includegraphics[width=35mm]{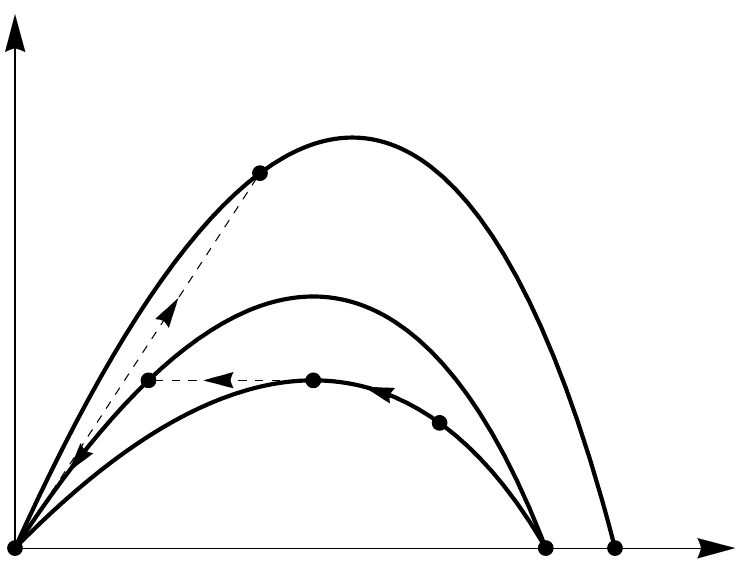}};
\node[below] at (35,0) {\strut $\rho$};
\node[left] at (-1,26.25) {\strut $c\,F_w$};
\node at (21,6) {\strut $\bF_L^-$};
\node at (8,19) {\strut $\bF_R^+$};
\node at (13,8) {\strut $\hat{\bF}^-$};
\draw[latex-] (7,8) --++ (8,-4) node[right] {\strut $\check{\bF}^+$};

\begin{scope}[shift={(37,0)}]
\node at (0,0) {\includegraphics[width=35mm]{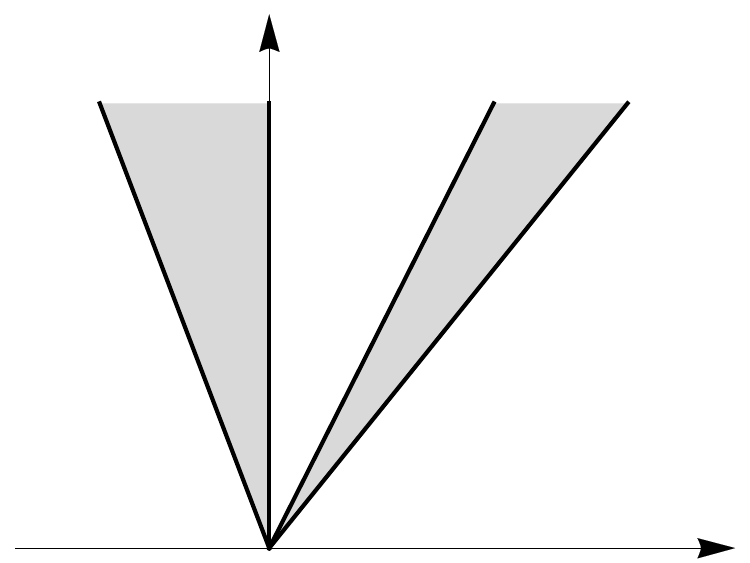}};
\node[below] at (35,0) {\strut $x$};
\node[left] at (10,26.25) {\strut $t$};
\node at (2,5) {\strut $\bW_L$};
\node at (26,7) {\strut $\bW_R$};
\draw[latex-] (12.5,22) --++ (-6,2) node[left] {\strut $\hat{\bW}$};
\node at (15,16) {\strut $\check{\bW}$};
\draw[latex-] (29.5,21.5) --++ (4,2) node[right] {\strut $\bf0$};

\end{scope}

\begin{scope}[shift={(74,0)}]
\node at (0,0) {\includegraphics[width=35mm]{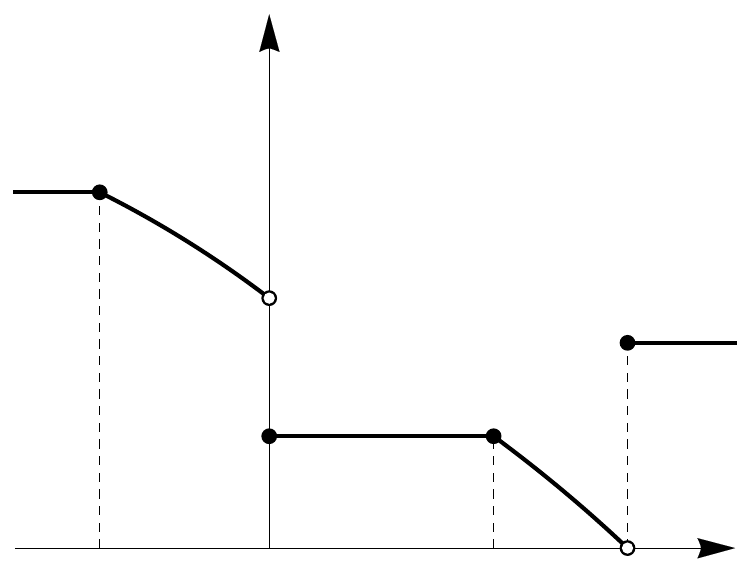}};
\node[below] at (35,0) {\strut $x$};
\node[below] at (5,0) {\strut $x_1$};
\node[below] at (24,0) {\strut $x_2$};
\node[below] at (30,0) {\strut $x_3$};
\node[left] at (10,26.25) {\strut $\rho$};
\node at (1,17.5) {\strut $\rho_L$};
\node at (30,10.5) {\strut $\rho_R$};
\node at (14,11) {\strut $\hat{\rho}$};
\node at (17,6) {\strut $\check{\rho}$};
\end{scope}

\begin{scope}[shift={(111,0)}]
\node at (0,0) {\includegraphics[width=35mm]{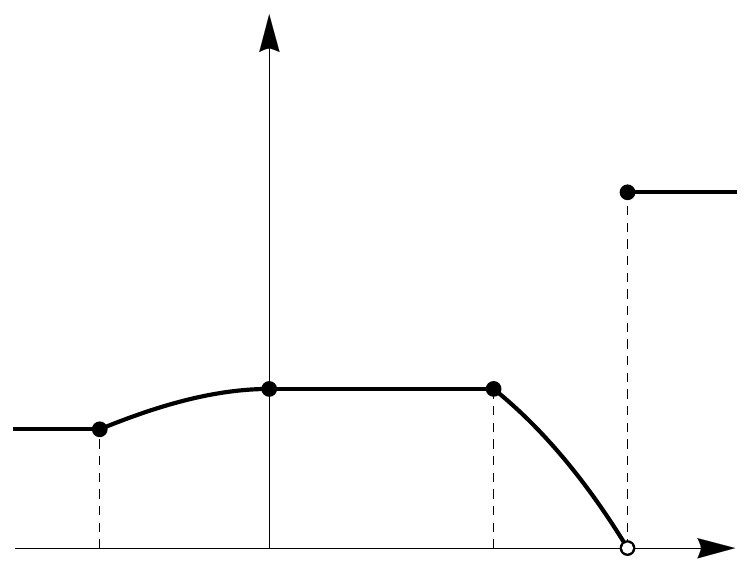}};
\node[below] at (35,0) {\strut $x$};
\node[below] at (5,0) {\strut $x_1$};
\node[below] at (24,0) {\strut $x_2$};
\node[below] at (30,0) {\strut $x_3$};
\node[left] at (10,26.25) {\strut $c\,F_w$};
\draw[latex-] (2,7) --++ (2,9) node {\strut $F_{w_L}^-(\rho_L)$};
\node[left] at (35,20) {\strut $F_{w_R}^+(\rho_R)$};
\node at (16,8) {\strut $\mathsf{Q}$};
\end{scope}

\begin{scope}[shift={(148,0)}]
\node at (0,0) {\includegraphics[width=35mm]{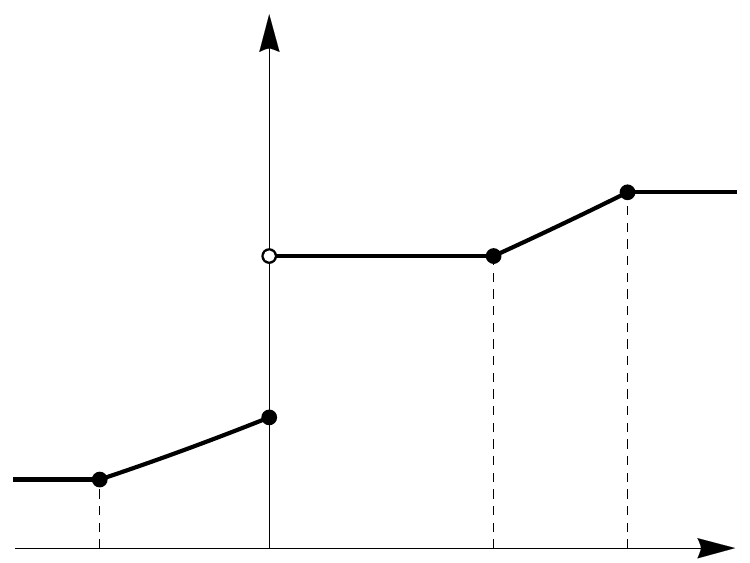}};
\node[below] at (35,0) {\strut $x$};
\node[below] at (5,0) {\strut $x_1$};
\node[below] at (24,0) {\strut $x_2$};
\node[below] at (30,0) {\strut $x_3$};
\node[left] at (10,26.25) {\strut $c\,V$};
\draw[latex-] (2,4) --++ (1.5,6) node {\strut $c^-V_L$};
\node at (27,17.5) {\strut $c^+V_R$};
\node at (13.5,5) {\strut $c^-\hat{V}$};
\node at (14,14) {\strut $c^+\check{V}$};
\end{scope}

\end{tikzpicture}
}
\caption{Representation of the weak solution $u$ in the case $p(\rho) = \left(\frac{1}{\rho}-1\right)^{-1}$, $V(v) = \sinh(v)$ and with the values given in \eqref{e:inivac}.
The profiles of $\rho$, $c\,F_w$ and $c\,V$ are computed at time $t=1$.
Above we take $V_L = V(v(\bW_L))$, $\hat{V} = V(v(\hat{\bW}(\bW_L,\bW_R)))$, $\check{V} = V(v(\check{\bW}(\bW_L,\bW_R)))$, ${\bf0} = (0,w_L)$, $\mathsf{Q}=\mathsf{Q}(\bW_L,\bW_R)$ and $F_w^\pm = c^\pm \, F_w$.
Moreover we let $x_1 = c^- \, \lambda_{w_L}(\rho_L) \approx -1.668$, $x_2 = c^+ \, \lambda_{w_L}(\check{\rho}(\bW_L,\bW_R)) \approx 2.208$ and $x_3 = c^+ \, \lambda_{w_L}(0) \approx 3.526$.}
\label{f:sol-Ber-Sinh}
\end{figure}
\begin{example}
\label{ex:3}
In \figurename~\ref{f:sol-Ber-Sinh} we show the construction of the weak solution $u$ in the case $p(\rho) = \left(\frac{1}{\rho}-1\right)^{-1}$, $V(v) = \sinh(v)$ and with
\begin{align}
\label{e:inivac}
&c^-=2,&
&c^+=3,&
&\rho_L = \frac{2}{5},&
&w_L = \uw \doteq 1,&
&w_R = \frac{13}{10},&
&\rho_R = p^{-1}\left(w_R-\ov^+\right) = \frac{3}{13},
\end{align}
where $\ov^+ = \arcsinh\left(\lambda_{w_L}(0)\right) = \arcsinh\left(\frac{e^2-1}{2 e}\right)$ and $e \approx 2.71828$ is Euler’s number.
In this case $\rho_L > \mathsf{r}_-(w_L)$, $\bW_*(\bW_R,w_L) = (0,w_L)$ and $\rho_*(\bW_R,w_L) < \mathsf{r}_+(w_L)$, therefore by definition, see \eqref{e:Qm}, \eqref{e:Qp} and \eqref{e:Q}, we have
\begin{align*}
&\mathsf{Q}_-(\bW_L) = \mathsf{F}_-(w_L),&
&\mathsf{Q}_+(\bW_R,w_L) = \mathsf{F}_+(w_L),&
&\mathsf{Q}(\bW_L,\bW_R) = c^- \, \mathsf{F}_-(w_L).
\end{align*}
By \ref{i:hat} and \ref{i:check} we have $\hat{\bW}(\bW_L,\bW_R) = ( \mathsf{r}_-(w_L)  , w_L)$ and $\check{\bW}(\bW_L,\bW_R) = ( F_{w_L}^{-1}(\frac{c^-}{c^+} \, \mathsf{F}_-(w_L)) , w_L )$.
Therefore, $u$ has a $1$-rarefaction in the cone $c^- \, \lambda_{w_L}(\rho_L) \leq x/t \leq c^- \, \lambda_{w_L}(\mathsf{r}_-(w_L)) = 0$, a stationary $c$-wave at $x=0$, a $1$-rarefaction in the cone $c^+ \, \lambda_{w_L}(\check{\rho}(\bW_L,\bW_R)) \leq x/t \leq c^+ \, \lambda_{w_L}(0)$ and a $2$-contact discontinuity at $x = c^+ \, V(v(\bW_R))$.
Observe that $V(v(\bW_R)) = \lambda_{w_L}(0) = V(\overline{v}^+)$.
\end{example}

\section{Numerical section}
\label{s:5}

\subsection{Glimm scheme motivation}

If $V$ is the identity and $c$ is constant, then $\RS_{\bar{c}}$ reduces, up to the time rescaling $t\mapsto \bar c \, t$, to the classical Riemann solver of the ARZ model, and the numerical difficulties encountered in simulations of the ARZ model are therefore inherited here. Namely, as observed in \cite{chalonsgoatin}, the domain of the conserved variable $\bU=(\rho,q)$ is not a convex region. Consequently, conservative numerical schemes that average $\rho$ and $q$, such as the Godunov scheme, fail to satisfy the maximum principle for the Riemann invariant $w$ (and, in this constant-$c$ regime, for $v$ as well). Typically, one observes spurious oscillations in the approximation of Riemann problems, whose solution reduces to a single contact discontinuity, that is, when $v$ is constant in the initial datum; see \cite{chalonsgoatin} for a detailed discussion. In \cite{chalonsgoatin} the authors propose a transport-equilibrium scheme that overcomes this difficulty and suggest two further routes to the same end: averaging $\rho$ and $v$ instead of $\rho$ and $q$, which, however, leads to considerably more complicated update formulas, or resorting to a Glimm-like scheme for contacts. In line with this approach—as applied to the ARZ model in \cite{Berthelin2017} and extended to general point constraints on the density flux in \cite{Donadello2026}—we adopt this second route. In these works, the Glimm scheme was shown to accurately capture solutions, including those involving vacuum formation.

When $c$ is not constant, a further difficulty appears which has no counterpart in the classical ARZ model. As already observed in the remarks following Definition~\ref{d:Haken}, even the \emph{exact} Riemann solver $\RS_{c^-,c^+}$ may fail to satisfy the maximum principle in the $v$-coordinate, and the total variation of $(v,w)$ may increase across a single Riemann problem whenever $c^-\neq c^+$, see Section~\ref{s:4} for explicit examples. This is not a numerical artifact but an intrinsic feature of the weak solutions of \eqref{e:system}, due to the interaction of a $w$-contact discontinuity with a $c$-wave. It is therefore essential to approximate \eqref{e:system} by means of a scheme that reproduces every Riemann solution \emph{exactly}, so that this genuine behaviour of the exact solution is not confused with spurious numerical oscillations of the kind described above. This further motivates our choice of the Glimm scheme.

\subsection{Glimm scheme description}
\label{sec:scheme}

Fix a space step $\Delta x>0$, introduce the points $x_{j+1/2} \doteq j \, \Delta x$, the cells $\mathcal{C}_j \doteq [x_{j-1/2},x_{j+1/2})$ and the cell centers $x_j \doteq (j-1/2) \, \Delta x$ for $j\in\Z$. Since $c\in\BV\cap\Lloc1(\R;(0,\infty))$, we approximate it by the piecewise constant function taking on $\mathcal{C}_j$ the value
\[c_j \doteq \frac{1}{\Delta x}\int_{\mathcal{C}_j} c(x) \, \d x.\]
For any integer $n\geq0$ and any $j\in\Z$, the goal is to compute $\bW_j^n \doteq (\rho_j^n,w_j^n)$, an approximation of $\bW(t^n,x_j)$, where $(t^n)_n$ is recursively defined by
\begin{align*}
&t^0 = 0,&
&t^{n+1} = t^n + \Delta t_{n+1},
\end{align*}
with the time step $(\Delta t_n)_n$ given below. Given an initial datum $\bW_o \in \Lloc1(\R;\calW)$, we start by setting $\bW_j^0 \doteq \bW_o(x_j)$.
For a fixed integer $n\geq0$, assuming that $\bW_j^n$ is given for any $j\in\Z$, we define the time step $\Delta t_{n+1}$ so that the following CFL condition is satisfied:
\[\Delta t_{n+1} = \frac{C_{\rm cfl} \, \Delta x}{S_{\max}^n},\]
where $S_{\max}^n$ denotes the maximum wave speed, in absolute value, present in the whole domain at time $t^n$; see Definition~\ref{d:Peryphery} with \eqref{e:lambdaw} and \eqref{e:s}, and the CFL coefficient $C_{\rm cfl}$ satisfies $0<C_{\rm cfl}\leq1/2$. We choose $C_{\rm cfl}=1/2$ for all the simulations of this paper. We use the following procedure to compute $\bW_j^{n+1}$ for any $j\in\Z$.

\begin{itemize}[leftmargin=*]
\item[$\bullet$] As in \cite{Collela_SIAJSSC_1982, Toro1997}, we choose a quasi-random number $\theta_{n}\in[0,1)$ by considering the van der Corput sequence $(\theta_{n})_{n\in\N}$ defined by $\theta_{n} = \sum_{k=0}^m i_k \, 2^{-(k+1)}$, where $(i_0,\ldots,i_m)$, with $i_k\in\{0,1\}$, is the binary expansion of the integer $n$, namely $n = \sum_{k=0}^m i_k \, 2^k$.
\item[$\bullet$] Define $\Delta\nu_{n+1} \doteq \Delta x/\Delta t_{n+1}$. Extending the notation of Definition~\ref{d:Haken} to the diagonal case by setting $\RS_{c^-,c^+} \doteq \RS_{\bar c}$ whenever $c^-=c^+=\bar c$, so that the two Riemann solvers of Definitions~\ref{d:Peryphery} and~\ref{d:Haken} are used consistently at every interface, the updated solution is then computed as
\begin{align*}
\bW_j^{n+1} \doteq
\begin{cases}
\RS_{c_{j-1},c_j}\left[\bW_{j-1}^n,\bW_j^n\right]\left(\theta_{n} \, \Delta\nu_{n+1}\right) &\hbox{if } 0\leq\theta_{n}<1/2,
\\[2pt]
\RS_{c_j,c_{j+1}}\left[\bW_j^n,\bW_{j+1}^n\right]\left((\theta_{n}-1) \, \Delta\nu_{n+1}\right) &\hbox{if } 1/2\leq\theta_{n}\leq1.
\end{cases}
\end{align*}
\end{itemize}

\noindent
Unlike the reconstruction step of the classical Godunov or transport-equilibrium schemes, this procedure never averages $\rho$ and $q$ (or $\rho$ and $v$) across a discretized $c$-wave: it merely samples, cell by cell, the exact self-similar solution of the local Riemann problem, exactly as in the Glimm scheme for the classical ARZ model. In particular, it does not introduce any numerical viscosity, so the size of jump discontinuities is computed exactly, at the price of locating them only ``up to one cell,'' an error which is negligible on fine meshes.

For the numerical simulations presented in the sequel, we work on a bounded computational domain $[x_{1/2},x_{N+1/2}]$, with $N$ mesh cells, and complement the scheme above with Neumann boundary conditions.

\subsubsection*{Remarks on the implementation of the method}

In the Riemann solver implementation, we could in principle pick different pairs of variables: the physical ones $(\rho,v)$, the conservative ones $(\rho,q)$, the Riemann invariants $(v,w)$, or even $(\rho,q)$ together with the auxiliary quantity $\mathsf{Q}$ introduced in \eqref{e:Q}. We perform our simulations in physical variables, as they are the easiest to compare with available data and can be translated into conservative variables or Riemann invariant coordinates in every situation, including at the vacuum. However, to avoid spurious numerical oscillations at the interfaces where $c_j\neq c_{j+1}$, our implementation keeps track of all the state variables $\rho$, $v$, $q$, $w$ and $p(\rho)$ simultaneously.

Indeed, by Definition~\ref{d:Haken} and the remark following it, the Riemann solver $\RS_{c_j,c_{j+1}}$ is obtained by conserving $w$ and equating the fluxes $c^\pm \, F_{w_\pm}(\rho_\pm)$ on both sides of the interface to the common value $\mathsf{Q}(\bW_L,\bW_R)$ given in \eqref{e:Q}. Floating-point round-off in the evaluation of $F_{w}(\rho) = \rho \, V(v(\rho,w))$ may therefore prevent the discrete states from meeting exactly this equality, which might make the scheme oscillate spuriously between the different branches of Definitions~\ref{d:hat} and~\ref{d:check} at each time step.

\subsection{Validation against exact solutions}

We first validate the scheme described in Section~\ref{sec:scheme} on the Riemann problems whose exact solutions have been constructed in Section~\ref{s:4}, namely Example~\ref{ex:1}~\ref{i:sol-log-sinh}, Example~\ref{ex:2}~\ref{i:sol-pow-id-inv} and Example~\ref{ex:3}.
In all three cases $c$ is the piecewise constant function given in \eqref{e:cRiemann}, the computational domain is $[-4,4]$, discretized with $N=256$ cells (so that $\Delta x = 1/32$), and the numerical solution is compared with the exact one at time $t=1$.
In \figurename~\ref{fig:SimEx1} we plot the density $\rho$ and the quantity $c\,V$, which is the actual velocity of the vehicles.

\begin{figure}[!htbp]
    \centering
    \begin{subfigure}[t]{0.31\textwidth}
        \includegraphics[width=1\linewidth]{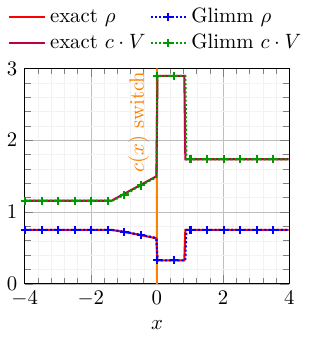}
        \caption{Example~\ref{ex:1}~\ref{i:sol-log-sinh}.}
        \label{fig:SimEx1-1}
    \end{subfigure}
    \hfill
    \begin{subfigure}[t]{0.31\textwidth}
        \includegraphics[width=1\linewidth]{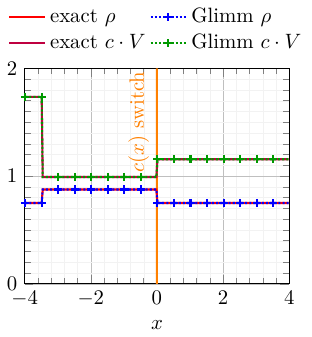}
        \caption{Example~\ref{ex:2}~\ref{i:sol-pow-id-inv}.}
        \label{fig:SimEx1-2}
    \end{subfigure}
    \hfill
    \begin{subfigure}[t]{0.31\textwidth}
        \includegraphics[width=1\linewidth]{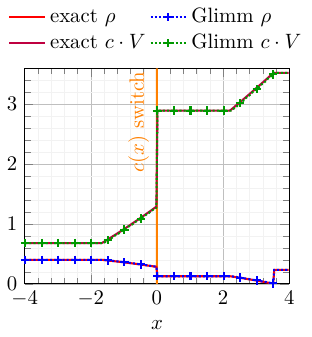}
        \caption{Example~\ref{ex:3}.}
        \label{fig:SimEx1-3}
    \end{subfigure}
    \caption{Comparison between the exact solution (solid lines) and the approximate solution (dotted lines with markers) at time $t=1$, with $N=256$ cells on $[-4,4]$.
    }
    \label{fig:SimEx1}
\end{figure}

The agreement between the approximate and the exact solutions is excellent in all three tests.
As expected from the Glimm scheme, the discontinuities are not smeared: the stationary $c$-wave at $x=0$, the $1$-shocks of Examples~\ref{ex:1} and~\ref{ex:2} and the $2$-contact discontinuity of Example~\ref{ex:3} are all captured with their exact amplitude, and their position is correct up to one cell.
The rarefaction waves are approximated by piecewise constant profiles which sample the exact solution at the cell centers.
No spurious oscillations appear at the interface $x=0$ where $c$ jumps, even though the exact solution itself violates the maximum principle in the $v$-coordinate there, see Examples~\ref{ex:1} and~\ref{ex:2}: the scheme reproduces this genuine increase of the total variation without introducing artificial oscillations.
Finally, \figurename~\ref{fig:SimEx1}\subref{fig:SimEx1-3} shows that the scheme captures the vacuum state $\rho=0$ which appears between the $1$-rarefaction and the $2$-contact discontinuity in Example~\ref{ex:3}, without producing negative densities.
Furthermore, a numerical convergence study performed on these three Riemann problems, see Appendix~\ref{app:convergence}, shows that the relative $\ell^1$-error of the scheme decreases with an approximately first-order rate with respect to the mesh size.

\subsection{Cauchy problem}

We now consider a Cauchy problem for \eqref{e:system} with $p(\rho)=\rho^3$, $V(v)=v$, the constant initial datum $\bW_o \doteq (0.75,1)$ already used in Example~\ref{ex:1}, and the periodic coefficient $c$ with period $1/2$ defined by
$$c(x) \doteq 1 - \brk*{x - \tfrac{1}{2}\lfloor 2x \rfloor},$$
namely $c$ decreases linearly from $1$ to $1/2$ on each interval $[k/2,(k+1)/2)$, $k\in\Z$, and jumps upward from $1/2$ to $1$ at every $x\in\frac{1}{2}\Z$.
The computational domain is $[-1,1]$, discretized with $N=1024$ cells, and the simulation is run up to time $t=1$.
The results are shown in \figurename~\ref{fig:Cauchy1a}.

\begin{figure}[!htbp]
    \centering
    \begin{subfigure}[t]{0.325\textwidth}\centering
        \includegraphics[height=35mm]{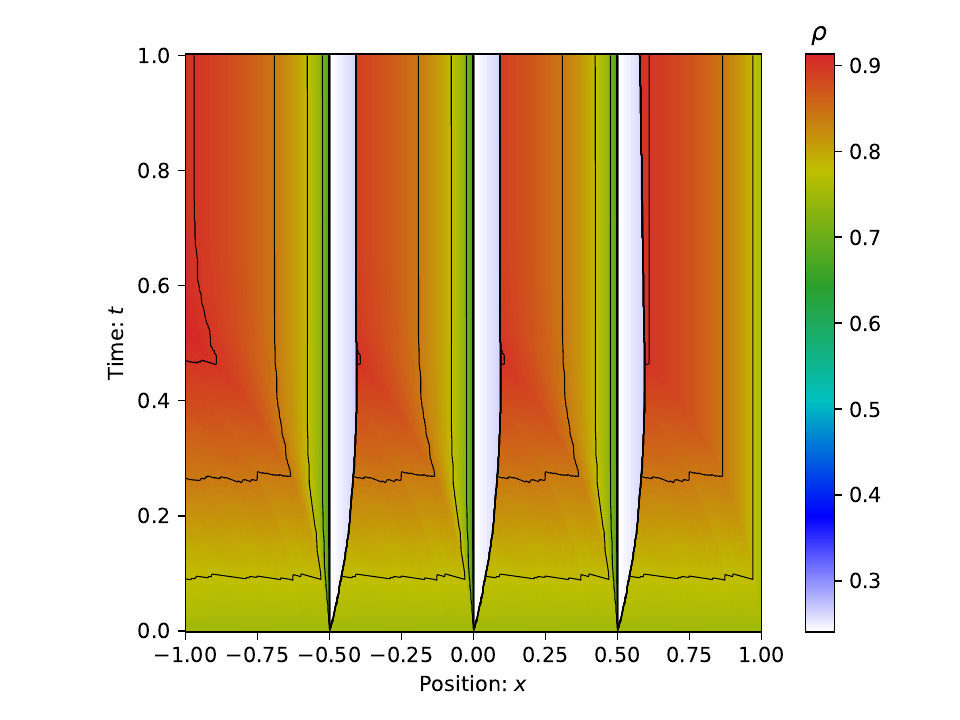}
        \caption{$\rho$ in the $(x,t)$-plane.}
        \label{fig:Cauchy1a-d}
    \end{subfigure}
    \begin{subfigure}[t]{0.325\textwidth}\centering
        \includegraphics[height=35mm]{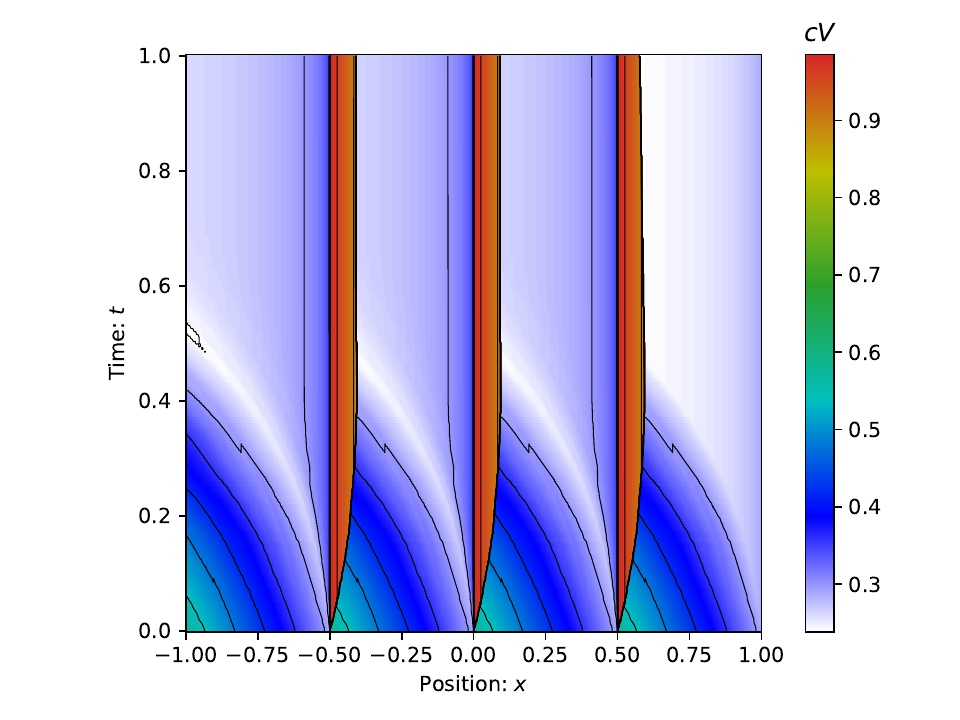}
        \caption{$c\,V$ in the $(x,t)$-plane.}
        \label{fig:Cauchy1a-v}
    \end{subfigure}
    \begin{subfigure}[t]{0.325\textwidth}\centering
        \includegraphics[height=35mm]{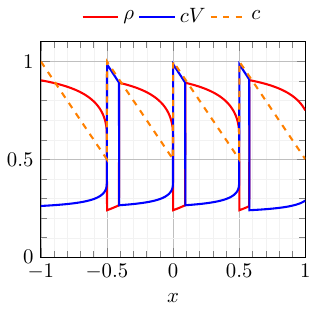}
        \caption{$\rho$, $c\,V$ and $c$ at $t=1$.}
        \label{fig:Cauchy1a-t1}
    \end{subfigure}
    \caption{Numerical solution of the Cauchy problem with $p(\rho)=\rho^3$, $V(v)=v$, constant initial datum $\bW_o=(0.75,1)$ and $c$ periodic with period $1/2$, decreasing linearly from $1$ to $1/2$ on each period. Glimm scheme with $N=1024$ cells on $[-1,1]$.}
    \label{fig:Cauchy1a}
\end{figure}

The initial datum is constant, and $w$ is preserved by the Glimm scheme so $w=w_o = 1$ for all times. Hence \eqref{e:system} reduces to the scalar conservation law $\partial_t\rho + \partial_x\left(c\,F_{w_o}(\rho)\right) = 0$ with discontinuous flux, and the dynamics is driven only by the spatial variations of $c$.
At each upward jump of $c$, the local Riemann problem is the one of Example~\ref{ex:1}~\ref{i:sol-log-sinh}, with $c^-=1/2$ and $c^+=1$.
Therefore, each such jump acts as a bottleneck whose capacity is $c^-\,\mathsf{F}_-(w_o)$, with $c^-=1/2$ and $w_o=1$. Upstream, where $c\,F_{w_o}(\rho_o)$, with $\rho_o=0.75$, exceeds this capacity, vehicles queue and the density increases up to about $0.9$, while it decreases towards the critical value $\mathsf{r}_-(w_o)=4^{-1/3}\approx0.63$ just before the jump. Downstream, a region of low density and high velocity appears, bounded on the right by a $1$-shock.
These shocks slow down and stop, see \figurename~\ref{fig:Cauchy1a}\subref{fig:Cauchy1a-d} and~\subref{fig:Cauchy1a-v}, so that the solution reaches a stationary regime before $t=1$, see \figurename~\ref{fig:Cauchy1a}\subref{fig:Cauchy1a-t1}.
In this regime the flux $c\,F_{w}(\rho)$ is constant in space and equal to $0.2367$, which coincides, up to discretization errors, with the capacity $c^-\,\mathsf{F}_-(w_o) \approx 0.2362$ of the bottlenecks, i.e.\ with the value of $\mathsf{Q}$ in Example~\ref{ex:1}~\ref{i:sol-log-sinh}.
Only the last period $[0.5,1]$ is affected by the Neumann boundary condition at $x=1$, where the flux keeps its initial value $c(1^-)\,F_{w_o}(\rho_o)\approx0.217$.


\providecommand{\bysame}{\leavevmode\hbox to3em{\hrulefill}\thinspace}
\providecommand{\MR}{\relax\ifhmode\unskip\space\fi MR }
\providecommand{\MRhref}[2]{%
	\href{http://www.ams.org/mathscinet-getitem?mr=#1}{#2}
}
\providecommand{\href}[2]{#2}

\adresa
Felisia Angela Chiarello, Department of Engineering, University LUM, S.S.~100 km~18, 70010 Casamassima, Italy,
e-mail: chiarello@lum.it

Simone Fagioli, Department of Engineering and Information Science and Mathematics, University of L’Aquila, Via Vetoio, Ed.~Coppito~1, 67100 Coppito, Italy,
e-mail: simone.fagioli@univaq.it

Bastien Polizzi, Laboratoire de Mathématiques de Besançon, Université Marie et Louis Pasteur, CNRS, LmB (UMR 6623), F-25000 Besançon, France,
e-mail: bastien.polizzi@math.cnrs.fr

Massimiliano Daniele Rosini, Faculty of Mathematics and Computer Science, Maria Curie-Skłodowska University, Plac Marii Curie-Skłodowskiej~1, 20031 Lublin, Poland, and Department of Management and Business Administration, University \lq\lq G.~d'Annunzio\rq\rq\ of Chieti-Pescara, viale Pindaro~42, 65127 Pescara, Italy,
e-mail: massimiliano.rosini@mail.umcs.pl


\appendix
\section{Numerical convergence}
\label{app:convergence}

To assess the numerical convergence of the Glimm scheme described in Section~\ref{sec:scheme}, we study the evolution of the relative $\ell^{1}$-error between the numerical solutions and the exact solutions of the three Riemann problems considered in \figurename~\ref{fig:SimEx1}, namely Example~\ref{ex:1}~\ref{i:sol-log-sinh}, Example~\ref{ex:2}~\ref{i:sol-pow-id-inv} and Example~\ref{ex:3}, for varying mesh sizes.
In all cases the computational domain is $[-4,4]$, $c$ jumps from $c^-$ to $c^+$ at $x=0$, and the number of mesh cells ranges from $N=2^6$ to $N=2^{18}$.

Using the notation introduced in Section~\ref{sec:scheme}, for any given quantity $u\in\{\rho,c\,V\}$, we denote by $u_{j}^{n}$ the approximate value on the cell $\mathcal{C}_{j}$ at time $t^{n}$ of $u(t^{n}, x_{j})$.
The relative $\ell^{1}$-error at time $t^n$ for $u$ is then
\begin{equation}
\label{e:FormulaError}
E_{\ell^{1}}^{n}(u) \doteq \Biggl(\sum_{j=1}^{N}\abs*{u\brk*{t^{n}, x_{j}}-u_{j}^{n}}\Biggr) \Big/ \Biggl(\sum_{j=1}^{N}\abs*{u\brk*{t^{n}, x_{j}}}\Biggr).
\end{equation}
The errors computed at time $t=1$ are reported in \figurename~\ref{fig:ConvEx}, together with their least-squares linear regression in logarithmic scale.
In all three tests the scheme converges with an approximately first-order rate $1$: the convergence rate is $0.93$ for both $\rho$ and $c\,V$ in Example~\ref{ex:1}~\ref{i:sol-log-sinh}, $1.00$ for both $\rho$ and $c\,V$ in Example~\ref{ex:2}~\ref{i:sol-pow-id-inv}, and $0.87$ for $\rho$ and $0.86$ for $c\,V$ in Example~\ref{ex:3}.
We also note that the solution of Example~\ref{ex:2}~\ref{i:sol-pow-id-inv} only consists of a stationary $c$-wave and a $1$-shock: for $N=256$ the computed error vanishes; hence this value is neither represented in \figurename~\ref{fig:ConvEx}\subref{fig:ConvEx-2} nor taken into account in the regression.

\begin{figure}[htbp]
    \centering
    \begin{subfigure}[t]{0.31\textwidth}
        \includegraphics[width=1\linewidth]{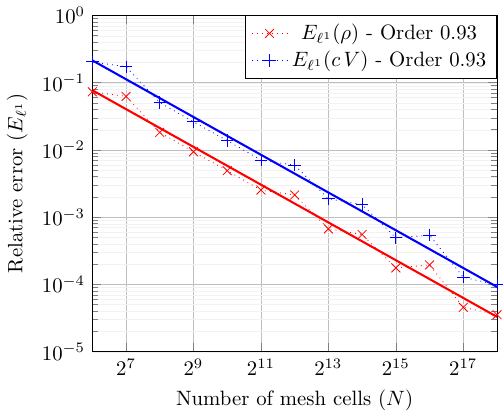}
        \caption{Example~\ref{ex:1}~\ref{i:sol-log-sinh}. The convergence rate is $0.93$ for $\rho$ and for $c\,V$.}
        \label{fig:ConvEx-1}
    \end{subfigure}
    \hfill
    \begin{subfigure}[t]{0.31\textwidth}
        \includegraphics[width=1\linewidth]{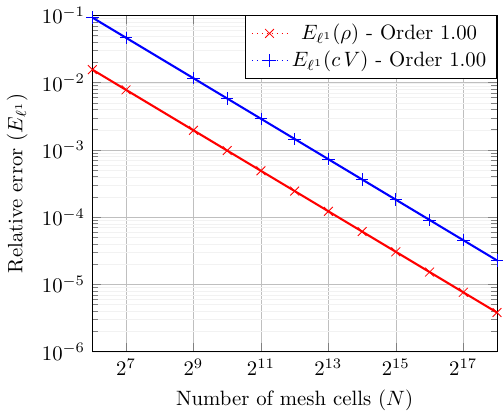}
        \caption{Example~\ref{ex:2}~\ref{i:sol-pow-id-inv}. The convergence rate is $1.00$ for $\rho$ and for $c\,V$.}
        \label{fig:ConvEx-2}
    \end{subfigure}
    \hfill
    \begin{subfigure}[t]{0.31\textwidth}
        \includegraphics[width=1\linewidth]{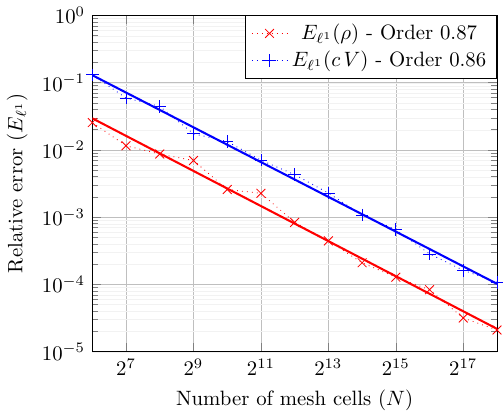}
        \caption{Example~\ref{ex:3}. The convergence rate is $0.87$ for $\rho$ and $0.86$ for $c\,V$.}
        \label{fig:ConvEx-3}
    \end{subfigure}
    \caption{Convergence analysis of the Glimm scheme: relative error in $\ell^{1}$-norm, see \eqref{e:FormulaError}, at time $t=1$ for the density $\rho$ and the velocity $c\,V$ with respect to the number of mesh cells $N$. Markers correspond to the computed errors and solid lines to their linear regression in logarithmic scale.}
    \label{fig:ConvEx}
\end{figure}
\end{document}